\documentclass[10pt]{article}

\usepackage[accepted]{tmlr}

\usepackage{amsmath,amsfonts,bm}

\def\eqref#1{equation~\ref{#1}}

\def\1{\bm{1}}

\DeclareMathAlphabet{\mathsfit}{\encodingdefault}{\sfdefault}{m}{sl}
\SetMathAlphabet{\mathsfit}{bold}{\encodingdefault}{\sfdefault}{bx}{n}

\usepackage{microtype}
\usepackage{graphicx}

\usepackage{booktabs}

\usepackage{hyperref}
\usepackage{indentfirst}
\usepackage{mathrsfs}
\usepackage{xcolor}

\usepackage{nicefrac}
\usepackage{url}
\usepackage{algorithm}

\usepackage{caption}
\usepackage{subcaption}
\usepackage{graphics}
\usepackage{tikz}
\usepackage[shortlabels]{enumitem}
\usepackage{mathtools}
\usetikzlibrary{positioning}
\usepackage{booktabs}
\usepackage{wrapfig}

\usepackage{comment}
\usepackage{color}
\usepackage{multirow}
\usetikzlibrary{shapes, arrows.meta, positioning}
\usetikzlibrary{backgrounds}
\usetikzlibrary{fit}

\usepackage{amsmath}
\usepackage{amssymb}
\usepackage{mathtools}
\usepackage{amsthm}
\usepackage{csquotes}

\theoremstyle{plain}

\theoremstyle{definition}

\theoremstyle{remark}

\numberwithin{equation}{section}
\newtheorem{theorem}{Theorem}
\newtheorem{lemma}{Lemma}

\newtheorem{proposition}{Proposition}

\newtheorem{definition}{Definition}

\allowdisplaybreaks[4]

\usepackage{authblk}

\author[1]{Chenyang Wu}
\author[2,3]{Qian Chen}
\author[2,4,*]{Akang Wang}
\author[2]{Tian Ding}
\author[2,4]{Ruoyu Sun}
\author[1,*]{Wenguo Yang}
\author[2,5]{Qingjiang Shi}

\affil[1]{University of Chinese Academy of Sciences, China}
\affil[2]{Shenzhen Research Institute of Big Data, China}
\affil[3]{School of Science and Engineering, The Chinese University of Hong Kong, Shenzhen, China }
\affil[4]{School of Data Science, The Chinese University of Hong Kong, Shenzhen, China}
\affil[5]{School of Computer Science and Technology, Tongji University, Shanghai, China}

\def\month{11}  
\def\year{2025} 
\def\openreview{\url{https://openreview.net/forum?id=GC2ZO6Asoa}} 

\begin{document}

\title{On Representing Convex Quadratically Constrained Quadratic Programs via Graph Neural Networks}
\maketitle

\footnotetext{*Corresponding authors: Wenguo Yang \textless yangwg@ucas.ac.cn\textgreater, Akang Wang \textless wangakang@sribd.cn\textgreater}

\begin{abstract}
Convex quadratically constrained quadratic programs (QCQPs) involve finding a solution within a convex feasible region defined by quadratic constraints while minimizing a convex quadratic objective function. These problems arise in various industrial applications, including power systems and signal processing. Traditional methods for solving convex QCQPs primarily rely on matrix factorization, which quickly becomes computationally prohibitive as the problem size increases. Recently, graph neural networks (GNNs) have gained attention for their potential in representing and solving various optimization problems such as linear programs and linearly constrained quadratic programs. In this work, we investigate the representation power of GNNs in the context of QCQP tasks. Specifically, we propose a new tripartite graph representation for general convex QCQPs and properly associate it with message-passing GNNs. We demonstrate that there exist GNNs capable of reliably representing key properties of convex QCQPs, including \textit{feasibility}, \textit{optimal value}, and \textit{optimal solution}. Our result deepens the understanding of the connection between QCQPs and GNNs, paving the way for future machine learning approaches to efficiently solve QCQPs.
\end{abstract}

\section{Introduction}
\label{sec:intro}

\textit{Quadratic programs}~(QPs) are a pivotal class of optimization problems where the objective function is quadratic, and the constraints are typically linear or quadratic.
Based on the nature of constraints, QPs can be further classified as \textit{linearly constrained quadratic programs}~(LCQPs) and \textit{quadratically constrained quadratic programs}~(QCQPs). 
When the objective and constraint matrices are positive semi-definite, the problem becomes a convex QCQP, making it both theoretically interesting and practically important. 
Convex QCQPs arise in various critical applications such as robust optimization in uncertain environments~\citep{ben2001lectures,boyd2004convex}, power flow~\citep{bienstock2020mathematical}, and signal processing~\citep{luo2010semidefinite}.

Solving QPs, especially those with quadratic constraints, presents significant challenges. 
Traditional methods often involve computationally intensive procedures that would struggle with scalability and real-time processing requirements. 
For example, the \textit{interior-point method}~\citep{nocedal1999numerical} for a general $n$-variable QP involves solving a sequence of linear systems of equations, necessitating matrix decomposition with a runtime complexity of~$\mathcal{O}(n^3)$. 
This leads to substantial computational burden in the large-scale case. 
Similarly, active-set algorithms~\citep{gill2019practical}, which work by iteratively adjusting the set of active constraints, can also become computationally demanding as the number of constraints and variables increase.

In recent years, advances in \textit{machine learning} (ML) have opened new avenues for enhancing the solving process of QPs. 
There are mainly two categories of ML-aided QP methods. 
The first category aims to learn adaptive configurations of a specific QP algorithm or solver to accelerate the solving process~\citep{bonami2018learning,ichnowski2021accelerating,jung2022learning}, while the second focuses on predicting an initial solution of QPs, which is either directly taken as a final solution or further refined by subsequent algorithms or QP solvers~\citep{bertsimas2022online,gao2021deep,sambharya2023end,tan2024ensemble,wang2020learning,xiongsolving}. Additionally, \cite{xiong2024neuralqp} proposes a hypergraph-based framework for solving QCQPs.
Most of these methods utilize \textit{graph neural networks}~(GNNs) to leverage the structural properties of graph-structured data, making them particularly well-suited for representing the relationships and dependencies inherent in QPs. 
By encoding QP instances into graphs, GNNs can capture intricate features and provide adaptive guidance or approximate solutions efficiently.

In addition to these studies, theoretical research on the expressive power of GNNs \citep{zhang2023expressive, li2022expressive} and their relation to optimization problems has further strengthened the understanding of their capabilities. 
For instance, \cite{chen2022LP} and \cite{chen2022MILP} established theoretical foundations for applying GNNs to solving \textit{linear programs}~(LPs) and \textit{mixed-integer linear programs}, respectively. 
Further, such foundations are extended to LCQPs and their discrete variant, mixed-integer LCQPs in \cite{chen2024expressive}.

Previous studies have empirically and theoretically demonstrated the utility of GNNs in speeding up existing QP solvers and directly approximating solutions for various QP instances. 
However, the question of \textit{whether GNNs can accurately predict key properties of QCQPs, such as feasibility, optimal objective value, and optimal solution}, remains open.
This paper aims to address the aforementioned gap by exploring both theoretical foundations and practical implementation of using GNNs for solving convex QCQPs. 
Specifically, we propose a tripartite graph representation for general convex QCQPs, and establish theoretical foundations of applying GNNs to optimize QCQPs. 
The distinct contributions of this paper can be summarized as follows.
\begin{itemize}
    
    \item \textbf{Graph Representation}. 
    We propose a novel tripartite graph representation for general QCQPs, which divides a QCQP into three types of nodes: linear-term, quadratic-term, and constraint nodes, with edges added between heterogeneous nodes to indicate problem parameters. 

    \item \textbf{Theoretical Foundation}. 
    We conduct analysis on the \textit{separation power} as well as \textit{approximation power} of \textit{message-passing GNNs}~(MP-GNNs). 
    We showed that MP-GNNs are capable of capturing some key properties of convex QCQPs.
    
    \item \textbf{Empirical Evidence}. 
    We conduct initial numerical tests of the tripartite MP-GNNs on small QCQP instances. The results showed that MP-GNNs can be trained to approximate the key properties well.
    
\end{itemize}

Building on the theoretical insights of our work, we establish a deeper understanding of the equivalence properties in convex QCQPs. By examining what key properties of QCQPs can be represented by GNNs, we gain a clearer view of QCQPs' intrinsic geometry and feasibility structure. 

Although our work is primarily theoretical, it opens up interesting opportunities for future empirical investigation. In particular, our insights can inform the design of practical GNN-based solvers that improve efficiency in tackling convex QCQPs. Moreover, while our main focus is on convex QCQPs, many solvers for more general, non-convex QCQPs rely on relaxation steps involving convex QCQPs at intermediate stages. Consequently, our insights could help accelerate or guide the solving of these relaxation-based subproblems, offering a pathway to enhanced performance even in non-convex scenarios.

\subsection*{Notations}
Throughout this paper, scalars or vectors are denoted by lowercase letters (e.g., $a$), and matrices are denoted by uppercase letters (e.g., $A$). For a vector $a$, we denote its $i$-th entry by $a_i$. For a matrix $A$, the entry in the $i$-th row and the $j$-th column is denoted by $a_{i,j}$. We use $\mathbf{0}$ and $\mathbf{1}$ to denote vectors or matrices with all-zero and all-one entries, respectively. For any positive integers $m,n$ with $m<n$, we define $[m,n] \coloneqq \{m, m+1,\cdots, n\}$ to be the set of all integers ranging from $m$ to $n$. For brevity, we define $[n] \coloneqq [1:n] = \{1, 2,\cdots, n\}$.

\section{Graph Representation of QCQPs} \label{sec:graph_representation}

\subsection{Quadratically Constrained Quadratic Programs} \label{sec:QCQP}

In this work, we study QCQPs defined in the following form:
\begin{equation}\label{eq:QCQP}
    \begin{aligned}
        & \min_{x \in \mathbb{R}^n} & & {f(x) \coloneqq} \frac{1}{2}x^\top Q x + p^\top x \\
        & \;\; \mathrm{s.t.} & & \frac{1}{2} x^\top Q^i x + (p^i)^\top x + b^i \leq 0 ~~~ \forall ~i\in [m]  \\
        &            && x^{\text{L}} \leq x \leq x^{\text{U}} 
    \end{aligned}
\end{equation}

where $Q, Q^i \in {\mathbb{S}^{n\times n}}$, $p,p^i \in \mathbb{R}^{n}$, $b^i\in \mathbb{R}$, $x^{\text{L}} \in (\mathbb{R} \cup \{-\infty\})^n$, and $x^{\text{U}} \in (\mathbb{R} \cup \{+\infty\})^n$. The problem has $n$ optimization variables and $m$ constraints. We refer to the tuple $(m,n)$ as the \emph{problem size} of QCQP. Both the objective function and the constraints are associated with quadratic functions. 

The QCQP problem is \textit{convex} if $Q$ and $Q^i$'s are all positive semi-definite.

{
We denote the \textit{feasible set} of Problem~(\ref{eq:QCQP}) by $\mathcal{X}$.}

If $\mathcal{X} \neq \emptyset$, the QCQP is said to be \textit{feasible}; otherwise, it is said to be \textit{infeasible}. A feasible QCQP is said to be \textit{bounded} if the objective is bounded from below on $\mathcal{X}$, i,e., there exists $z \in \mathbb{R}$ such that ${f(x)} \geq z$ for every $x \in \mathcal{X}$; otherwise, it is said to be \textit{unbounded}. For a feasible and bounded QCQP, $x^* \in \mathcal{X}$ is said to be an optimal solution if {$f(x^*) \leq f(x)$}

for every $x \in \mathcal{X}$.
We remark that a QCQP always admits an optimal solution if it is feasible and bounded, but such an optimal solution might not be unique.

\subsection{Tripartite Representation of QCQPs}  \label{sec:tripartite_representation}

The first theoretical result demonstrating the representation power of GNNs in solving optimization problems was provided by~\cite{chen2022LP}. 
{Their work employs a bipartite graph representation where variables and constraints are modeled as nodes. In this encoding, the linear constraint coefficients are edge features, the right-hand-side values are constraint node features, and the objective coefficients are variable node features.}
They showed that GNNs based on this graph representation can universally approximate the optimal solution of LPs, as well as properties of feasibility and boundedness. 
This bipartite graph modeling was later extended {by encoding quadratic objective terms as edge features between variable nodes} to analyze the representation power of GNNs for LCQPs~\citep{chen2024expressive}.

Despite these advances, it remains challenging to develop graph representation to encode all information of general QCQPs while maintaining simplicity for GNN processing. Due to the presence of quadratic terms, a QCQP generally involves $\mathcal{O}(n^2 \times m)$ coefficients. Consequently, a graph encoding all QCQP information inherently exhibits a complexity of the same order, $\mathcal{O}(n^2 \times m)$. There are two natural extensions of the traditional bipartite representation of LPs/LCQPs to QCQPs.

\begin{itemize}
    \item[--] {\textbf{Hyperedge Representation}. This approach adds hyperedges to the traditional bipartite graph to represent quadratic coefficients, turning the graph into a hypergraph. However, to the best of our knowledge, most of the current GNN architectures struggle to handle hyperedges. 
    
    }

    \item[--] {\textbf{Vector Feature Representation}. In this {extension}, all coefficients are encoded as features associated with the $n$ variable nodes and the $m$ constraint nodes, resulting in a graph with vector features of varying sizes, depending on the problem. However, existing GNNs are generally incapable of processing features of varying dimensions.}
\end{itemize}

To fill this gap, we introduce an undirected \textit{tripartite graph} representation $G_{\mathrm{QCQP}} \coloneqq (V, E)$ that encodes all elements of a QCQP~(\ref{eq:QCQP}). Compared to the traditional bipartite graph modeling for LPs and LCQPs, our tripartite graph representation introduces an additional class of nodes to model the quadratic terms of variables. This modification allows us to represent QCQPs without any loss of information. In this paper, we will show that the tripartite representation enables GNNs to universally approximate solutions for convex QCQPs.

{
Formally, we model a QCQP as a tripartite graph with node sets representing variables, quadratic terms, and constraints.
\begin{itemize}
    \item \textbf{Variable nodes ($V_1$):} Let $V_1 \coloneqq \{u_1, \dots, u_n\}$, where node $u_j$ corresponds to variable $x_j$ and is associated with the feature tuple $(p_j, x^{\mathrm{L}}_j, x^{\mathrm{U}}_j)$.
    \item \textbf{Quadratic term nodes ($V_2$):} Let $\mathcal{L} \coloneqq \{(j, k) \in [n] \times [n] : j \leq k,\ |q_{j,k}| + \sum_{i=1}^m |q^i_{j,k}| > 0\}$ be the set of quadratic term indices with non-zero coefficients in either the objective or at least one constraint. Then, $V_2 \coloneqq \{v_{j,k} : (j, k) \in \mathcal{L}\}$. The feature for node $v_{j,k}$ is $2q_{j,k}$ if $j > k$, and $q_{j,j}$ if $j = k$.
    \item \textbf{Constraint nodes ($V_3$):} Let $V_3 \coloneqq \{c_1, \dots, c_m\}$, where node $c_i$ corresponds to the $i$-th constraint and is associated with the feature $b_i$.
\end{itemize}
The full node set is $V \coloneqq V_1 \cup V_2 \cup V_3$.~\footnote{Throughout, we use $u$ for variable nodes, $v$ for quadratic term nodes, and $c$ for constraint nodes. Indices $i$, $j$, and $k$ refer to constraints, variables, and quadratic terms, respectively.}
The QCQP graph contains three edge sets, each with associated weights:
\begin{itemize}
    \item \textbf{$V_1$--$V_2$ edges ($E_{12}$):} Connect variable node $u_{j'} \in V_1$ to quadratic node $v_{j,k} \in V_2$ if $j' = j$ or $j' = k$. The weight is defined as $w_{u_{j'}, v_{j,k}} \coloneqq 1$ if $j > k$, and $2$ otherwise.
    \item \textbf{$V_1$--$V_3$ edges ($E_{13}$):} Connect $u_j \in V_1$ to constraint node $c_i \in V_3$ if the linear coefficient $p^i_j \neq 0$. The weight is defined as $w_{u_j, c_i} \coloneqq p^i_j$.
    \item \textbf{$V_2$--$V_3$ edges ($E_{23}$):} Connect $v_{j,k} \in V_2$ to $c_i \in V_3$ if the quadratic coefficient $q^i_{j,k} \neq 0$. The weight is defined as $w_{v_{j,k}, c_i} \coloneqq 2q^i_{j,k}$ if $j > k$, and $q^i_{j,j}$ otherwise.
\end{itemize}
The full edge set is given by $E \coloneqq E_{12} \cup E_{13} \cup E_{23}$. For any edge, we define $w_{x,y} \coloneqq w_{y,x}$ for connected nodes $x$ and $y$.
}

\begin{figure*}[t]
    \centering
    \includegraphics[width=0.9\textwidth]{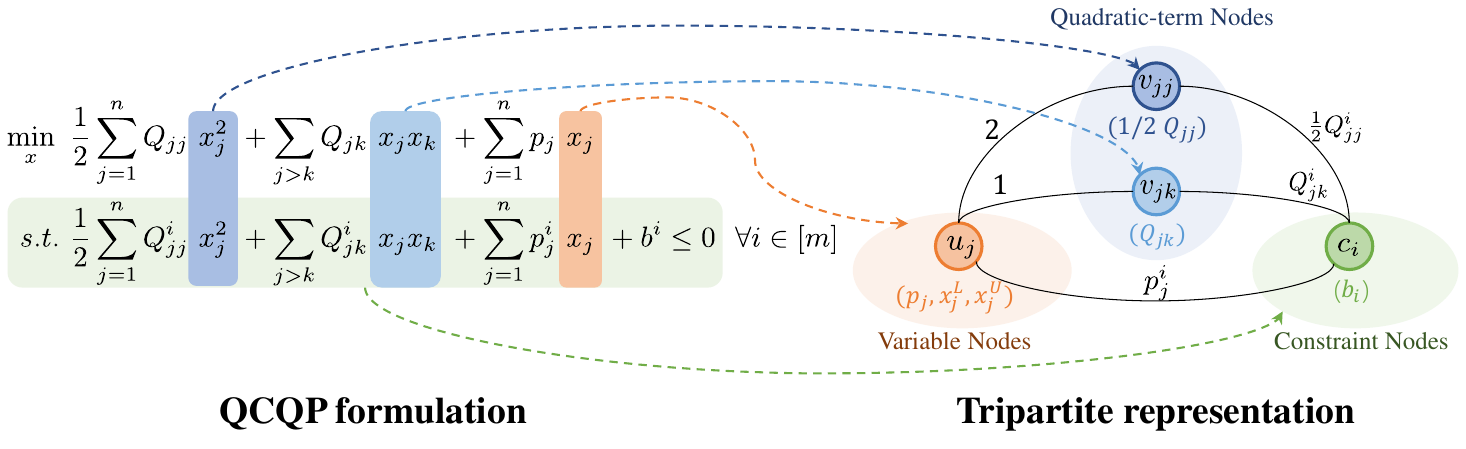}
    \caption{A tripartite representation of QCQPs. It consists of three types of nodes: variable nodes, quadratic-term nodes, and constraint nodes. All nodes and the edges connecting them are associated with coefficients from the formulation as features.}
    \label{fig:tripartite_rep}
\end{figure*}

We illustrate this representation in Figure~\ref{fig:tripartite_rep}. We remark that there is a one-to-one mapping between a QCQP and its tripartite graph representation $G_{\mathrm{QCQP}}$.
Our representation has two major advantages:

\begin{itemize}
    \item[--] {Our representation forms a simple graph, making it compatible with most existing GNNs. This facilitates the development of more efficient and task-specific GNN architectures.}
    \item[--] {Our representation can leverage the sparsity nature of QCQP problems. That is, the numbers of nodes and edges are also controlled by the total number of nonzero coefficients. Thus, our representation can be efficiently applied to large and sparse instances.}
\end{itemize}

\begin{definition}[Spaces of Convex QCQP-graphs]
    We denoted by 
    
    $\mathcal{G}^{m,n}_{\mathrm{QCQP}}$ the set of tripartite graph representations for all \textbf{convex} QCQPs with $n$ variables and $m$ constraints. \footnote{For any QCQP graph in $\mathcal{G}^{m,n}_{\mathrm{QCQP}}$, the associated convex QCQP can be characterized by its coefficient tuple {$(Q,\{Q^i\}^m_{i=1},p,\{p^i\}^m_{i=1},\{b^i\}^m_{i=1},x^\mathrm{L},x^\mathrm{U})$, where $Q, Q^i \in \mathbb{S}^n_+$}. We define a topology on $\mathcal{G}_{\mathrm{QCQP}}$: for $Q,Q^i$ and $p,p^i$ we use the topology induced by the norm of the linear mappings defined by the matrices and vectors, and for $x^{\mathrm{L}},x^{\mathrm{U}},b^i$ we use euclidean topology on $\mathbb R$ and discrete topology on the infinite values. In numerical experiments, we represent the infinite values by introducing an extra infinity indicator.}
\end{definition}

\section{Theoretical Results}

\subsection{Tripartite MP-GNNs}
\label{sec:theory_pre}
To study the capability of GNNs in representing QCQPs, we tailor the general MP-GNNs for the tripartite nature of the introduced QCQP graph representation. An overview is depicted in Figure \ref{fig:GNN}.

Specifically, we consider the family of tripartite MP-GNNs consisting of an embedding layer, $T$ message-passing layers (each comprised of four sub-layers), and a readout layer, detailed as follows:

\begin{figure}[t]
    \centering
    \includegraphics[width=0.8\textwidth]{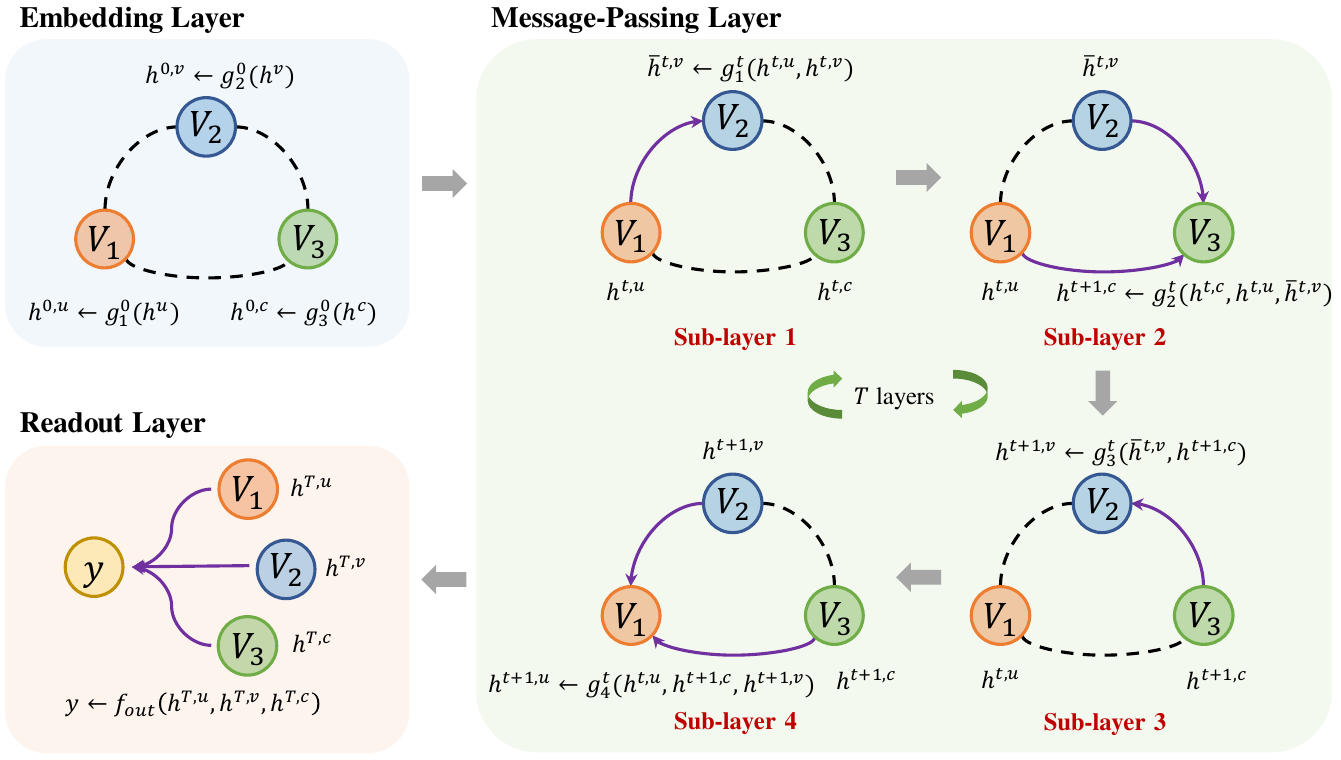}
    \caption{{An overview of the GNN architecture.}}
    \label{fig:GNN}
\end{figure}

\begin{itemize}

    \item \textbf{Embedding Layer.} The initial features for all nodes are obtained by projecting their input features into a hidden space $\mathbb{R}^{h_0}$ using learnable embedding functions:
\[
    \begin{aligned}
        h^{0,u} &\gets g^0_1(h^u) \quad \forall u\in V_1,\\
        h^{0,v} &\gets g^0_2(h^v) \quad \forall v\in V_2,\\
        h^{0,c} &\gets g^0_3(h^c) \quad \forall c\in V_3,
    \end{aligned}
\]
where $h^u$, $h^v$, and $h^c$ are the input features of nodes in $V_1$, $V_2$, and $V_3$, respectively, and $g^0_1, g^0_2, g^0_3$ are the embedding functions for each node type.

    \item \textbf{Message-Passing Layer.} Each message-passing layer employs four distinct sub-layers to update node features, using learnable functions \( f^t_l \) and \( g^t_l \). Each sub-layer updates the features of nodes in one set by aggregating messages from a specific subset of neighboring nodes, as detailed below.

\begin{itemize}
    \item \textbf{Sub-layer 1: Update Quadratic Nodes from Variables (\(V_1 \rightarrow V_2\))}:
        \begin{equation*}
            \bar h^{t,v} \gets g_1^t\left(h^{t,v},\sum_{u \in V_1} w_{u,v} f_1^t(h^{t,u})\right) \quad \forall v \in V_2.
        \end{equation*}
    This step computes an intermediate representation \(\bar h^{t,v}\) for each quadratic node \(v\) based on incoming messages from variable nodes \(u \in V_1\).

    \item \textbf{Sub-layer 2: Update Constraint Nodes from Variables and Quadratic Nodes (\(V_1 + V_2 \rightarrow V_3\))}:
        \begin{equation*}
            h^{{t+1},c} \gets g_2^t\left(h^{t,c},m^t_{13},m^t_{23}\right) \quad \forall c\in V_3,
        \end{equation*}
        where \(m^t_{13} \coloneqq \sum_{u\in V_1} w_{u,c}f_2^t(h^{t,u})\) and \(m^t_{23} \coloneqq \sum_{v\in V_2} w_{v,c}f_3^t(\bar h^{t,v})\) denote the messages from variable nodes and quadratic nodes, respectively. This sub-layer produces the updated constraint node features for the next step.

    \item \textbf{Sub-layer 3: Update Quadratic Nodes from Constraints (\(V_3 \rightarrow V_2\))}:
        \begin{equation*}
            h^{{t+1},v} \gets g_3^t\left(\bar h^{t,v},\sum_{c\in V_3} w_{c,v}f_4^t(h^{{t+1},c})\right)\quad \forall v\in V_2.
        \end{equation*}
        Here, the intermediate representations \(\bar h^{t,v}\) are updated using the new constraint node features \(h^{{t+1},c}\) to produce the final quadratic node features for this layer.

    \item \textbf{Sub-layer 4: Update Variable Nodes from Constraints and Quadratic Nodes (\(V_3 + V_2 \rightarrow V_1\))}:
        \begin{equation*}
            h^{{t+1},u} \gets g_4^t\left(h^{t,u},m^t_{31},m^t_{21}\right) \quad \forall u\in V_1,
        \end{equation*}
        where \(m^t_{31} \coloneqq \sum_{c\in V_3} w_{c,u}f_5^t(h^{{t+1},c})\) and \(m^t_{21}\coloneqq \sum_{v\in V_2} w_{v,u}f_6^t(h^{{t+1},v})\) are the messages from constraint nodes and the updated quadratic nodes, respectively. This completes the update cycle for the layer.
\end{itemize}

The four sub-layers implement a structured, alternating update scheme. The first two sub-layers propagate information \textit{from} variable nodes \textit{to} constraint nodes, passing through the quadratic nodes to incorporate information about quadratic terms. This allows the constraint node representations to be informed by the current state of the variables and their quadratic interactions.

Conversely, the last two sub-layers propagate information in the reverse direction, \textit{from} the updated constraint nodes \textit{to} the variable nodes. By again routing messages through the quadratic nodes, the network learns the relationship between constraints and quadratic terms, using this to refine the variable node representations. This bidirectional, structured message-passing is crucial for capturing the complex dependencies when representing QCQPs.

    \item \textbf{Readout Layer.} The readout layer maps the final node features from the $T$-th message-passing layer to an output $y \in \mathbb{R}^s$ using a learnable function $f_{\mathrm{out}}$. We consider two output types:

\begin{itemize}
    \item \textbf{Graph-level output} ($s=1$): The output is a scalar representing the entire graph. We compute it as:
    \begin{equation*}
        y = f_{\mathrm{out}}\left(a_1, a_2, a_3\right),
    \end{equation*}
    where $a_1 = \sum_{u \in V_1} h^{T,u}$, $a_2 = \sum_{v \in V_2} h^{T,v}$, and $a_3 = \sum_{c \in V_3} h^{T,c}$ are the aggregated features from variable, quadratic, and constraint nodes, respectively.
    
    \item \textbf{Node-level output} ($s=n$): The output is a vector where each component corresponds to a variable node. For each node $u_j \in V_1$:
    \begin{equation*}
        y_j = f_{\mathrm{out}}\left(h^{T,u_j}, a_{1,u_j}, a_2, a_3\right), \quad \forall j \in [n]
    \end{equation*}
    where $a_{1,u_j} = \sum_{u \in V_1 \setminus \{u_j\}} h^{T,u}$ aggregates features from all other variable nodes, providing contextual information about the graph excluding node $u_j$ itself.
\end{itemize}
\end{itemize}

  \color{black}

\begin{definition}[Spaces of GNNs]
\label{def:GNN_space}
    Let $\mathcal{F}_{\mathrm{QCQP}}(\mathbb{R}^s)$ denote the collection of all tripartite MP-GNNs, parameterized by continuous embedding functions {$g^0_{l}, l \in \{1,2,3\}$}, continuous hidden functions in the message passing layers {$g^t_{l}$ with $l \in \{1,2,3,4\}$, $h^t_{l}$ with $l \in \{1,2,3,4,5,6\}$}, and the continuous readout function $f_{\mathrm{out}}$. Specifically, for a given problem size $(m,n)$ of QCQP, there exists a subset of GNNs in $\mathcal{F}_{\mathrm{QCQP}}(\mathbb{R}^s)$ that maps the input space $\mathcal{G}^{m,n}_{\mathrm{QCQP}}$ to the output space $\mathbb{R}^s$. This subset of GNNs are denoted by $\mathcal{F}^{m,n}_{\mathrm{QCQP}}(\mathbb{R}^s)$.
\end{definition}

We define the following target functions, characterizing some key properties on learning an end-to-end network to predict the optimal solutions of convex QCQPs:

\begin{definition}[Target mappings]
\label{def:QCQP_properties}
Let $G_{\mathrm{QCQP}}$ be a tripartite graph representation of a QCQP problem. We define the following target mappings.
\begin{itemize}
    \item \textit{Feasibility mapping}:
    We define $\Phi_{\mathrm{feas}}(G_{\mathrm{QCQP}}) = 1$ if the QCQP problem is feasible and $\Phi_{\mathrm{feas}}(G_{\mathrm{QCQP}}) = 0$ otherwise.
        
    \item \textit{Boundedness mapping}: for a feasible QCQP problem, we define $\Phi_{\mathrm{bound}}(G_{\mathrm{QCQP}}) = 1$ if the QCQP problem is bounded and $\Phi_{\mathrm{bound}}(G_{\mathrm{QCQP}}) = 0$ otherwise.
                 
    \item \textit{Optimal value mapping}: for a feasible and bounded QCQP problem, we set $\Phi_{\mathrm{opt}}(G_{\mathrm{QCQP}})$ to be its optimal objective value.

    \item \textit{Optimal solution mapping}: for a feasible, bounded QCQP problem, there must exist at least an optimal solution, but the optimal solution might not be unique. However, if the QCQP is convex, there exists a unique optimal solution $x^*$ with the smallest $\ell_2$-norm among all optimal solutions. Therefore, for a \textit{convex} QCQP we define the optimal solution mapping to be $\Phi_{\mathrm{sol}}(G_{\mathrm{QCQP}})$ = $x^*$. Since the optimal solution with the smallest $\ell_2$-norm may not be unique for non-convex QCQPs, we do not define its optimal solution mapping.\footnote{In fact, Section \ref{sec:non_convex_counter_examples} shows that there exists a pair of non-convex QCQPs that cannot be distinguished by any {MP-GNNs based on tripartite graph representation}. Thus, even if an optimal solution mapping for non-convex QCQPs is defined, MP-GNNs cannot universally approximate it.}
    \end{itemize}
\end{definition}

\subsection{Universal approximation for convex QCQPs}

{
Our theoretical analysis rests on the following key lemma, which establishes that the tripartite WL-test can transfer solutions between equivalent QCQP instances.
\begin{lemma} \label{lemma:solution_passing}
    Let $\mathcal{I},\Bar{\mathcal{I}}$ (with given sizes $m,n$, encoded by $G,\Bar{G}\in\mathcal{G}^{m,n}_{\mathrm{QCQP}}$) be two QCQP instances. If
    the tripartite WL-test cannot separate $G$ from $\Bar{G}$, then for any feasible solution $x$ of \(\mathcal{I}\), there exists a feasible solution \(\Bar{x}\) for \(\Bar{\mathcal{I}}\) such that:
    \begin{itemize}
        \item[(i)] $\frac{1}{2} \Bar{x}^\top\Bar{Q}\Bar{x}+\Bar{p}^\top \Bar{x} \leq \frac{1}{2} x^\top Qx + p^\top x$;
        \item[(ii)] $\|\Bar{x}\|_2  \leq \|x\|_2$.
    \end{itemize}
\end{lemma}
Lemma~\ref{lemma:solution_passing} resolves a fundamental challenge. In LPs/LCQPs, solution transfer between equivalent instances can be achieved through straightforward averaging of variable values within node color classes. For QCQPs, this direct approach fails because the average of products does not equal the product of averages.
Our core contribution is proving that the \textit{stable colorings} generated by the tripartite WL-test provide sufficient structural information to overcome this limitation. The solution $\Bar{x}$ is constructed by averaging the values of the original solution $x$ across all variables belonging to the same color class in $G$. While this averaging does not preserve individual quadratic terms, we prove that the \textit{collective quadratic forms}--the weighted sums of quadratic terms within each constraint--are appropriately controlled due to the color stability of the graph representation. This ensures that feasibility is maintained in $\Bar{\mathcal{I}}$ and the objective value is non-increasing.
This lemma establishes the fundamental separation power of the tripartite WL-test: it can distinguish between QCQP instances that have different optimal values or solution structures.
}

Building on this foundation, we demonstrate that for convex QCQPs, any target function in Definition \ref{def:QCQP_properties} can be universally approximated by MP-GNNs. Formally, we have the following theorem.

\begin{theorem}\label{thm:approximation} 
For any probability measure $\mathbb{P}$ on the space of convex QCQPs $\mathcal{G}^{m,n}_{\mathrm{QCQP}}$  and any $\delta, \varepsilon > 0$, there exists $F \in \mathcal{F}^{m,n}_{\mathrm{QCQP}}(\mathbb{R}^s)$ such that for any target mapping $\Phi:\mathcal{G}^{m,n}_{\mathrm{QCQP}}\to\mathbb{R}^s$ defined in Definition~\ref{def:QCQP_properties}, we have
    \begin{equation}
        \mathbb{P}\left\{||F(G_{\mathrm{QCQP}})-\Phi(G_{\mathrm{QCQP}})|| > \delta \right\} < \varepsilon.  \label{eq:main_theorem}
    \end{equation}
\end{theorem}

Theorem \ref{thm:approximation} highlights that sufficiently expressive GNNs can predict the feasibility, boundedness, optimal value, and optimal solution for convex QCQP problems with an arbitrarily small error. The proof of Theorem \ref{thm:approximation} is provided in Appendix~\ref{sec:detailed_proof}.

{The convexity requirement is essential for our theoretical results. Convexity ensures that the averaged solution $\Bar{x}$ constructed in Lemma~\ref{lemma:solution_passing} not only remains feasible but also satisfies the critical objective value bound in condition (i). This property enables the translation of the WL-test's separation power into the GNN's approximation capability for the target functions $\Phi$.}

\subsection{MP-GNNs can not represent general non-convex QCQPs}
\label{sec:non_convex_counter_examples}

In contrast to convex QCQPs, MP-GNNs based on tripartite graph representation do not possess universal representation power for non-convex QCQPs. Formally, we have the following propositions.

\begin{proposition}\label{prop:can_not_represent_non_convex}
    There exists non-convex QCQP instances $\mathcal{I},\bar{\mathcal{I}}$ encoded by tripartite graph representation $G,\Bar{G}$ respectively, such that 
    $\Phi(G)_{\text{feas}}\not=\Phi_{\text{feas}}(\Bar{G})$, but any GNN $F\in\mathcal{F}_{\mathrm{QCQP}}(\mathbb{R})$ gives $F(G)=F(\Bar{G})$.
\end{proposition}

\begin{proposition}\label{prop:can_not_represent_non_convex2}
    There exists non-convex QCQP instances $\mathcal{I},\bar{\mathcal{I}}$ encoded by tripartite graph representation $G,\Bar{G}$ respectively, such that
    \begin{itemize}
        \item[(i)] $\Phi(G)_{\text{opt}}\not=\Phi_{\text{opt}}(\Bar{G})$;
        \item[(ii)] the optimal solution sets of $\mathcal{I}$ and $\bar{\mathcal{I}}$ do not intersect;
        \item[(iii)] any GNN $F\in\mathcal{F}_{\mathrm{QCQP}}(\mathbb{R})$ gives $F(G)=F(\Bar{G})$.
    \end{itemize}
\end{proposition}

Proposition \ref{prop:can_not_represent_non_convex} implies that GNNs cannot universally predict the feasibility of non-convex QCQPs. Proposition \ref{prop:can_not_represent_non_convex2} implies that GNNs can neither universally predict the optimal value nor the optimal solution of non-convex QCQPs. We prove both propositions by constructing counter-examples. Below we present the counter-example for Proposition \ref{prop:can_not_represent_non_convex2}. We defer the formal proof of both propositions to Appendix \ref{sec:GNN_not_enough_separation_power_on_non_convex_QCQPs_proof}.

Consider the following pair of non-convex QCQPs:

{
\begin{equation}\label{eq:non_convex_instance_1}
    \begin{aligned}
        & \underset{x \in 
    \mathbb{R}^6}{\min}  & & x_1x_2+x_2x_3+x_3x_1+x_4x_5+x_5x_6+x_6x_4 \\
        & \;\; \text{s.t.} & & \sum_i x_i^2\le 1 \\
    \end{aligned}
\end{equation}
\begin{equation}\label{eq:non_convex_instance_2}
    \begin{aligned}
        & \underset{x \in 
        \mathbb{R}^6}{\min}  & & x_1x_2+x_2x_3+x_3x_4+x_4x_5+x_5x_6+x_6x_1 \\
        & \;\; \text{s.t.} & & \sum_i x_i^2\le 1 \\
    \end{aligned}
\end{equation}
}
\normalsize

\begin{figure*}[t]
    \centering
    \includegraphics[width=0.8\textwidth]{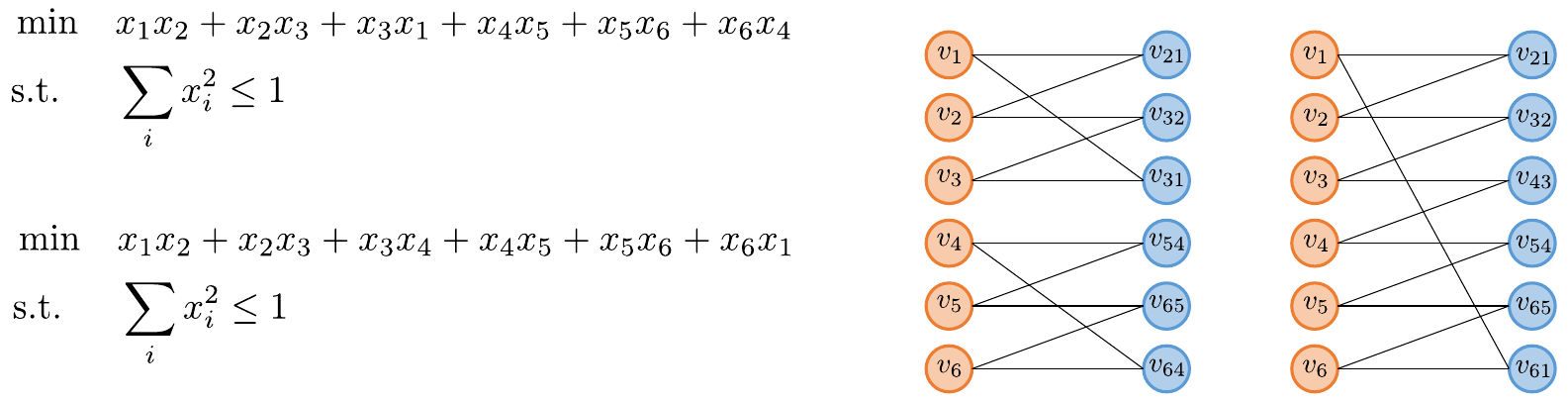}
    \caption{Left: two QCQP instances for proving Prop. \ref{prop:can_not_represent_non_convex}. Right: Parts of the corresponding tripartite graph representations to show the difference.}
    \label{fig:non_convex_examples}
\end{figure*}

{
For the former, the optimal objective value is $\Phi_{\mathrm{obj}}=-\frac12$, and all optimal solutions are given by 
\begin{equation*}
    \left\{x \in \mathbb{R}^6: x_1+x_2+x_3=0,x_4+x_5+x_6=0,\sum_i x_i^2=1\right\}.
\end{equation*}
For the latter, the optimal objective value $\Phi_{\mathrm{obj}}=-1$, and all optimal solutions are given by 
\begin{equation*}
    \left\{x \in \mathbb{R}^6:x_1=x_3=x_5=-x_2=-x_4=-x_6=\pm\frac{\sqrt{6}}{6}\right\}. 
\end{equation*}
}

We see that the optimal values of Problem~(\ref{eq:non_convex_instance_1}) and Problem~(\ref{eq:non_convex_instance_2}) are different, and their optimal solution sets do not intersect. The tripartite graph representations of the two instances are illustrated in Figure~\ref{fig:non_convex_examples}. We will further demonstrate in appendix \ref{sec:GNN_not_enough_separation_power_on_non_convex_QCQPs_proof} that any GNNs on the two 
tripartite graphs gives the same output. Thus, Problem~(\ref{eq:non_convex_instance_1}) and Problem~(\ref{eq:non_convex_instance_2}) serve as a valid counter-example for proving Proposition \ref{prop:can_not_represent_non_convex2}.

{We remark that Propositions~\ref{prop:can_not_represent_non_convex} and~\ref{prop:can_not_represent_non_convex2} demonstrate limitations specific to our proposed tripartite MP-GNN architecture. They do not rule out the potential for other, more powerful graph representations or network architectures to achieve universal approximation for non-convex QCQPs.}

\section{Computational Experiments}

In this section, we present empirical experiments to validate the proposed theoretical results. The corresponding source code is available at \href{https://github.com/NetSysOpt/L2QP}{https://github.com/NetSysOpt/L2QP}.

\subsection{Learning tasks}
In Theorem \ref{thm:approximation}, we established that there exists a function $F \in \mathcal{F}^{m,n}_{\mathrm{QCQP}}(\mathbb{R}^s)$ capable of approximating the target mapping $\Phi$ with an arbitrarily small error. To empirically confirm this claim, we design three supervised learning tasks to find such functions $F_{\text{feas}}$, $F_{\text{obj}}$ and $F_{\text{sol}}$, which are responsible for predicting feasibility, objective values, and optimal solutions, respectively.  For each task, let $\{(G_{i},y_i)\}_{i=1}^N$ be a given dataset, where $G_i$ represents a QCQP instance and $y_i$ denotes its corresponding label. The function family $\mathcal{F}^{m,n}_{\mathrm{QCQP}}(\mathbb{R}^s)$ is constructed using the tripartite MP-GNNs as defined in Definition \ref{def:GNN_space}. With all these ingredients ready, the learned function is obtained by  $F = \arg\min_{f\in \mathcal{F}^{m,n}_{\mathrm{QCQP}}(\mathbb{R}^s)} \frac{1}{N}\sum_{i=1}^N L\left(f(G_i),y_i\right)$, where $L(\cdot,\cdot)$ is the loss function. {Specifically, we use mean squared error for predicting objective values and optimal solutions, while binary cross-entropy loss is employed for predicting feasibility. We chose cross-entropy loss for this binary classification task as it is more suitable than a regression-based approach that would be needed if we were to predict numerical degrees of constraint violation instead of a final feasibility state.}

\subsection{Data generation}
To facilitate the supervised learning approach described above, we generate three datasets of convex QCQPs by randomly sampling coefficients from normal distributions. These datasets consist of a general QCQP (GP), a constrained least squares problem (CLS), and a trust-region subproblem (TRS). The formulation and generation strategy for each dataset are outlined below.

\paragraph{GP:} A general QCQP formulation is given in (\ref{eq:QCQP}). All coefficients of $Q$, $p$, $Q^i$, $p^i$, and $b^i$ in the objective and constraints are independently sampled from the normal distribution $\mathcal{N}(0,1)$. The lower and upper bounds $x^L$ and  $x^U$ are set to be $\textbf{0}$ and $\textbf{1}$, respectively. The number of variables and constraints for instances in this dataset are set to 30 and 5, respectively. We generate a small dataset for general QCQPs because larger-scale instances with more constraints tend to be predominately infeasible. 

\paragraph{CLS:} Constrained least squares problems with quadratic constraints are common in various applications. We examine the following CLS problem:
\begin{equation}
    \begin{aligned}
    \min_{x\in\mathbb{R}^n}~~ &\|Ax-b\|^2 \\
    \text{s.t.} ~~&x^\top Q x \leq c
    \end{aligned}
\end{equation}
The elements in $A$, $b$ and $Q$ are sampled from the standard normal distribution $\mathcal{N}(0,1)$, while elements in $c$ are sampled from $\mathcal{N}(1,1)$.  For this dataset, the number of variables is set to 500, and the sparsity (the proportion of zero elements) of $Q$ is set to 0.95.

\paragraph{TRS:} The trust-region subproblem is another form of convex QCQP, defined by the following formulation:
\begin{equation}
    \begin{aligned}
        \min_{x\in \mathbb{R}^n} ~~&x^\top Qx + 2c^\top x \\
        \text{s.t.} ~~&\|x\|^2\leq \Delta^2
    \end{aligned}
\end{equation}
The coefficients in $Q$ and $c$ are independently sampled from the standard normal distribution $\mathcal{N}(0,1)$ and $\Delta$ is sampled from $\mathcal{N}(1,1)$. Like the CLS dataset, the number of variables is set to 500, and the sparsity of $Q$ is set to 0.95.

To ensure that the generated instances remain convex, we adjust all sampled matrices $Q$ and $Q_i$s corresponding to the quadratic terms in both the objective function and the constraints. Specifically, each matrix is modified by replacing it with $Q - \alpha I$, where $\alpha <0$ is the minimal eigenvalue of $Q$. This adjustment ensures that the matrices are positive semi-definite, thus guaranteeing the convexity of the corresponding QCQP instances.

For each dataset, we generate 1,000 instances for training and 300 instances for validation. All instances are solved using the IPOPT solver~\citep{wachter2006implementation}. The resulting feasibility, objective values, and optimal solutions are collected as labels.

\subsection{GNN architecture and training settings} 
For the GNN described in Section \ref{sec:theory_pre}, there are three classes of functions $\{g_1^t,\dots,g_4^t\}_{t=1}^T$, $\{h_1^t,\dots,h_6^t\}_{t=1}^T$ and $R$ remain unspecified. 
The first class, $\{g_1^t,\dots,g_4^t\}_{t=1}^T$, are two-layer MLPs with layer widths of $[d,d]$, and ReLU as activations, where the inputs of each function are concatenated together. The second class, $\{h_1^t,\dots,h_6^t\}_{t=1}^T$, are linear transformations with output dimension $d$ followed by ReLU activations. The last one, $R$, is also a two-layer MLP with ReLU activation, with widths of $[d,1]$ for predicting feasibility and objective values, and $[d,n]$ for predicting solutions. The hyper-parameter $T$ is set to 2.
For training, we utilized the Adam optimizer with a learning rate of 0.0001 and a batch size of 16.

\subsection{Main results}
\label{sec:main-results}

In this section, we present the main results for tasks with different learning targets. For the GP dataset, we perform tasks to predict the feasibility, objective value, and optimal solution. However, the feasibility task is omitted for the CLS and TRS datasets, as their sampled instances are predominantly feasible.

\paragraph{Training loss vs. numbers of parameters.}

The results from Table \ref{table:tr_loss_vs_model_size} show the training loss for models with different numbers of parameters. Generally, the losses are small across all models, validating the claim in Theorem \ref{eq:main_theorem}. Moreover, we observe a consistent trend: as the number of parameters increases, the training loss decreases.

\begin{table*}[h]
\centering
\caption{Training loss vs. numbers of parameters.}
\scalebox{1}{
\begin{tabular}{cccccccc}
\toprule
\multirow{2}{*}{Dataset} & \multirow{2}{*}{Target} & \multicolumn{6}{c}{\# Parameters}          \\ \cmidrule{3-8} 
                         &                         & 15K    & 21K    & 42K    & 126K   & 1.7M   & {6.7M}     \\
                         \midrule
\multirow{3}{*}{GP}      & feasibility             & 0.2114 & 0.0992 & 0.0883 & 0.0779 & 0.0254   & {0.0044} \\
                         & objective               & 0.0853 & 0.0293 & 0.0218 & 0.0185 & 0.0110   & {0.0034} \\
                         & solution                & 0.0604 & 0.0419 & 0.0413 & 0.0407 & 0.0397   &{0.0350}  \\
                         \midrule
\multirow{2}{*}{CLS}     & objective               & 0.0994 & 0.0196 & 0.0135 & 0.0096 & 0.0020   & {0.0016} \\
                         & solution                & 0.2333 & 0.1454 & 0.0175 & 0.0173 & 0.0171   & {0.0169} \\
                         \midrule
\multirow{2}{*}{TRS}     & objective               & 2.1098 & 0.1979 & 0.1089 & 0.0996 & 0.0265   & {0.0171} \\
                         & solution                & 1.2004 & 0.4294 & 0.4035 & 0.4022 & 0.4017   & {0.4010}\\
                         \bottomrule
\end{tabular}
}
\label{table:tr_loss_vs_model_size}
\end{table*}

For the GP dataset, the feasibility, objective, and solution losses all decrease as the number of parameters increases, with the loss for the largest model (6.7M parameters) being particularly small, especially for the objective and feasibility targets. For the CLS dataset, similar trends are observed, with the loss for the objective function decreasing substantially from 0.0994 (15K parameters) to 0.0016 (6.7M parameters). The solution loss in CLS remains relatively stable but still decreases as the number of parameters increases. Finally, in the TRS dataset, both the objective and solution losses follow a similar pattern, with the objective loss improving significantly from 2.1098 (15K parameters) to 0.0171 (6.7M parameters).

\paragraph{Validation loss vs. number of samples.}

The validation loss results presented in Table \ref{table:val_loss_vs_nsample} indicate that the validation loss remains small across all configurations, highlighting the generalization capability of the models. Additionally, as the number of training samples increases, the validation loss decreases, reflecting the benefit of having more data for model training.

\begin{table*}[h]
\centering
\caption{Validation loss vs. numbers of training samples.}
\scalebox{1}{
\begin{tabular}{ccccccc}
\toprule
\multirow{2}{*}{Dataset} & \multirow{2}{*}{Target} & \multicolumn{5}{c}{\# Samples}             \\ \cmidrule{3-7} 
                         &                         & 100    & 300    & 500    & 700    & 1,000   \\
                         \midrule
\multirow{3}{*}{GP}      & feasibility             & 0.5857 & 0.3167 & 0.1940 & 0.1885 & 0.1004 \\
                         & objective         & 0.1067 & 0.0577 & 0.0514 & 0.0421 & 0.0296 \\
                         & solution        & 0.1149 & 0.0439 & 0.0438 & 0.0434 & 0.0427 \\
                         \midrule
\multirow{2}{*}{CLS}     & objective         & 0.0136 & 0.0121 & 0.0107 & 0.0086 & 0.0049 \\
                         & solution        & 0.0181 & 0.0179 & 0.0178 & 0.0177 & 0.0177 \\
                         \midrule
\multirow{2}{*}{TRS}     & objective         & 0.2131 & 0.0933 & 0.0834 & 0.0792 & 0.0596 \\
                         & solution        & 0.4081 & 0.4055 & 0.4054 & 0.4054 & 0.4054\\
\bottomrule
\end{tabular}}
\label{table:val_loss_vs_nsample}
\end{table*}

For the GP dataset, the validation loss for feasibility, objective, and solution all decrease as the number of samples increases. Notably, the solution loss stabilizes after 500 samples, while the objective and feasibility losses continue to improve as the number of samples grows, with the best performance observed at 1,000 samples. The CLS dataset shows a clear reduction in the objective loss as the number of samples increases, with the loss dropping from 0.0136 for 100 samples to 0.0049 for 1,000 samples. The solution loss in CLS remains relatively stable across different sample sizes. Finally, for the TRS dataset, the objective loss improves from 0.2131 at 100 samples to 0.0596 at 1,000 samples, while the solution loss remains nearly constant across all sample sizes.

In conclusion, both the training and validation loss results highlight the effectiveness of the models, showing that increasing the number of parameters and training samples leads to improved performance in terms of both training and generalization.

\subsection{Additional comparative results}

\paragraph{Runtime comparison.}

Table~\ref{table:runtime} reports the average solving time (in seconds) and standard deviation across the validation instances for several QCQP solvers—{IPOPT}~\citep{wachter2006implementation}, {Gurobi}~\citep{gurobi}, MOSEK~\citep{mosek}, SCS~\citep{ocpb:16}, and {ECOS}~\citep{domahidi2013ecos}—together with the inference time of our GNN. The results demonstrate that our GNN-based approach delivers a substantial speedup, often by orders of magnitude, compared to all baseline solvers. { However, it is important to note that the accuracy of these baseline solvers ($\leq 10^{-6}$) significantly surpasses that of the GNN, as reported in Section \ref{sec:main-results} (with a minimal level of $10^{-3}$), indicating a trade-off between speed and precision.}

\begin{table}[h]
\centering
\caption{Comparison of runtime (in seconds) with solvers having a tolerance of $10^{-6}$.}
\scalebox{1}{

\begin{tabular}{lccc}
\toprule
Method & GP                & GLS               & TRS               \\
\midrule
IPOPT   & 0.0374( ± 0.0044) & 0.0602( ± 0.0115) & 0.0670( ± 0.0158) \\
GUROBI  & 0.0054( ± 0.0004) & 0.1091( ± 0.0138) & 1.8921( ± 0.2751) \\
MOSEK   & 0.0300( ± 0.0033) & 0.4014( ± 0.0274) & 0.2264( ± 0.0079) \\
SCS     & 0.0311( ± 0.0058) & 0.6799( ± 1.0300) & 0.1466( ± 0.8167) \\
ECOS    & 0.0314( ± 0.0040) & 4.4056( ± 0.4668) & 1.6713( ± 0.5785) \\
GNN     & 0.0024( ± 0.0032) & 0.0022( ± 0.0030) & 0.0023( ± 0.0031) \\
\bottomrule
\end{tabular}
}
\label{table:runtime}
\end{table}

{{

\paragraph{Impact of sparsity.}

We conducted an ablation study on datasets CLS and TRS by varying sparsity (fraction of zero coefficients) from 0.95 to 0.8. Table~\ref{table:sparsity} reports the training losses.
\begin{table}[h]
\centering
\caption{Training loss on different levels of sparsity. }

\begin{tabular}{llcccc}
\toprule
\multirow{2}{*}{Dataset} & \multirow{2}{*}{Target} & \multicolumn{4}{c}{Sparsity}      \\ \cmidrule{3-6} 
                         &                         & 0.95   & 0.9    & 0.85   & 0.8    \\
                         \midrule
\multirow{2}{*}{CLS}     & objective               & 0.0020 & 0.0036 & 0.0058 & 0.0096 \\
                         & solution                & 0.0171 & 0.0217 & 0.0279 & 0.0323 \\
                         \cmidrule{2-6}
\multirow{2}{*}{TRS}     & objective               & 0.0265 & 0.0497 & 0.1084 & 0.1844 \\
                         & solution                & 0.4017 & 0.5338 & 0.7415 & 0.9042\\
\bottomrule
\end{tabular}
\label{table:sparsity}
\end{table}
As sparsity decreases, losses increase consistently, reflecting the added difficulty of denser problems with more complex interactions. Our focus on highly sparse cases (e.g., 0.95) is motivated by their prevalence in practice—for example, over 35\% of QPLib \citep{furini2019qplib} instances exhibit sparsity above 0.95.

In addition to comparisons on runtime and sparsity, we also provide evaluations on a real-world dataset, QPLIB, and on a task of predicting boundedness, which can be found in Appendix \ref{appendix:add_res}.
}}

\subsection{{Discussions}} \label{sec:discussion_results}

{
The computational efficiency of our tripartite GNN stems from two key design choices. First, the graph structure naturally exploits problem sparsity, as edges correspond directly to non-zero parameters, ensuring high efficiency for real-world, sparse instances. Second, the architecture uses simple message-passing layers where the split-layer design organizes computations without increasing the total number of matrix operations, incurring minimal overhead. We remark that our universal approximation theorem establishes expressive power but does not quantify the parameters required for a given approximation error; deriving such bounds remains an important direction for future work.}

\section{Conclusions}
This paper introduces a new tripartite graph representation specifically designed for QCQPs. By leveraging the capabilities of MP-GNNs, this approach shows theoretical promise in predicting key properties of QCQPs with arbitrary desired accuracy, including feasibility, boundedness, optimal values, and solutions. Initial numerical experiments validate the effectiveness of our framework.
This work advances learning-to-optimize by extending GNNs to QCQP problems, which have been challenging for traditional graph-based L2O methods. 
Our findings may inspire future research into more specialized GNN architectures and principled approaches to controlling GNN sizes for practical QCQP applications, going beyond the basic GCN structure used here.

\section*{Acknowledgments}
This work was supported by the National Key R\&D Program of China under grant~2022YFA1003900.
Qian Chen and Akang Wang also acknowledge support from National Natural Science Foundation of
China (Grant No. 12301416), Guangdong Basic and Applied Basic Research Foundation (Grant
No. 2024A1515010306), and Shenzhen Science and Technology Program (Grant No.~JCYJ20250604191330040). 

\bibliography{main}
\bibliographystyle{tmlr}

\newpage
\appendix
\onecolumn
\section{Detailed proof of main theorem}\label{sec:detailed_proof}

\subsection{Sketch of the proof}

We provide a brief outline of this proof:
\begin{itemize}
    \item[(i)] \textbf{Separation Power of WL-Test:} We first establish that the WL-test has sufficient separation power on the defined target functions.
    \item[(ii)] \textbf{Connection to tripartite MP-GNNs:} We then demonstrate the relationship between the separation power of tripartite MP-GNNs and that of the Tripartite WL-tests, showing that the GNNs can separate our target functions. This result, combined with the generalized Weierstrass theorem, leads to our approximation power conclusions.
    \item[(iii)] \textbf{Universal Approximation:} Assuming the target functions are continuous and have compact support, we prove universal approximation. In this step, we also specify the problem size and apply the Generalized Weierstrass Theorem (Theorem 22 of ~\cite{azizian2020expressive}).
    \item[(iv)] \textbf{Addressing Discontinuities:} Since the target functions are neither continuous nor compactly supported, particularly at the boundary of the convex QCQPs universe $\mathcal{G}^{m,n}_{\mathrm{QCQP}}$, we construct a continuous approximation of the target function to apply universal approximation, ensuring convergence in measure.
\end{itemize}

\subsection{WL-test on tripartite graph representation}
Here we describe our Tripartite WL-test, which is the WL-test counterpart of the tripartite MP-GNNs:

\begin{itemize}
    \item \textbf{Embedding.} Initial colors $C^{0,u}$, $C^{0,v}$, and $C^{0,c}$ are assigned based on their corresponding features and node types (e.g., from $V_1$, $V_2$, or $V_3$):
        \begin{itemize}
            \item $C^{0,u} \gets \mathrm{HASH}_1(f(u))$ for $u\in V_1$,
            \item $C^{0,v}  \gets \mathrm{HASH}_2(f(v))$ for $v\in V_2$,
            \item $C^{0,c} \gets \mathrm{HASH}_3(f(c))$ for $c\in V_3$.
        \end{itemize}
        Here, we refer to the color of a node after the $t$-th message-passing layer as $C^{t,\cdot}$.
    
    \item \textbf{Update quadratic nodes via variable nodes} ($V_1 \to V_2$):
        \[
            \bar C^{t,v} \gets \mathrm{HASH}\left(C^{t,v},\sum_{u \in V_1} w_{u,v} \mathrm{HASH}(C^{t,u})\right), \forall v \in V_2
        \]
    
    \item \textbf{Update constraint nodes via variable and quadratic nodes} ($V_1, V_2 \to V_3$):
        \[
            C^{{t+1},c} \gets \mathrm{HASH}\left(C^{t,c},\sum_{u \in V_1} w_{u,c} \mathrm{HASH}(C^{t,u}),\sum_{v \in V_2} w_{v,c} \mathrm{HASH}(\bar C^{t,v})\right), \forall c \in V_3
        \]
    
    \item \textbf{Update quadratic nodes again via constraint nodes} ($V_3 \to V_2$):
        \[
            C^{{t+1},v} \gets \mathrm{HASH}\left(\bar C^{t,v},\sum_{c \in V_3} w_{c,v} \mathrm{HASH}(C^{{t+1},c})\right), \forall v \in V_2
        \]
    
    \item \textbf{Update variable nodes via constraint and quadratic nodes} ($V_3, V_2 \to V_1$):
        \[
            C^{{t+1},u} \gets \mathrm{HASH}\left(C^{t,u},\sum_{c \in V_3} w_{c,u} \mathrm{HASH}(C^{{t+1},c}), \sum_{v \in V_2} w_{v,u} \mathrm{HASH}(C^{{t+1},v})\right), \forall u \in V_1
        \]
    
    \item \textbf{Termination and Readout.} Once a termination condition is met, we return the color collection \((C^{T,u})_{u \in V_1}, (C^{T,v})_{v \in V_2}, (C^{T,c})_{c \in V_3}\).\footnote{Multiple occurrences of members are counted instead of rejected.}
    
    \item All hash functions are real-valued and assumed to be collision-free.
\end{itemize}

In this paper, we terminate the Tripartite WL-test only when the algorithm stabilizes\footnote{For simplicity, we exclude the final iteration showing that the algorithm has stabilized and return the last iteration in which stabilization occurred.}, i.e., when the number of distinct colors no longer changes in an iteration (after all four color updates). Despite not imposing a forced iteration limit, the WL-test is guaranteed to terminate in a finite number of iterations, denoted by $T$:

\begin{proposition}[Tripartite WL-test terminates in finite iterations]\label{prop:wl_finite_iterations}
    The Tripartite WL-test stabilizes in a finite number of iterations.
\end{proposition}

\begin{proof}
    It is straightforward to observe from the formulation that if two nodes have different colors, they will continue to have different colors after an (sub-)iteration. Therefore, the number of iterations required for stabilization is capped by the number of distinct nodes, which is finite.
\end{proof}

We say that the Tripartite WL-test \textit{separates} two graphs if the resulting collection of colors differs between the two graphs. We claim that the Tripartite WL-test has the same separation power as its network counterpart, specifically the tripartite MP-GNNs:

\begin{proposition}[Tripartite MP-GNNs have equal separation power as the Tripartite WL-Test]\label{prop:message_passing_GNN_has_equal_separation_power}
    Given two instances $\mathcal{I}$ and $\Bar{\mathcal{I}}$ (correspondingly encoded by graphs $G$ and $\Bar{G}$), the following holds:
    \begin{itemize}
        \item[(i)] For graph-level output cases, the two instances are separated by $\mathcal{F}_{\mathrm{QCQP}}^{m,n}(\mathbb{R})$, i.e., 
        \begin{equation*}
            F(G) = F(\Bar{G}), \forall F \in \mathcal{F}_{\mathrm{QCQP}}^{m,n}(\mathbb{R})
        \end{equation*}        
        if and only if the two instances are also separated by the Tripartite WL-test.
        
        \item[(ii)] For node-level output cases, i.e., $\mathbb{R}^s = \mathbb{R}^n$, the two instances are separated by $\mathcal{F}_{\mathrm{QCQP}}^{m,n}(\mathbb{R})$, i.e.,
        \begin{equation*}
            F(G) = F(\Bar{G}), \forall F \in \mathcal{F}_{\mathrm{QCQP}}^{m,n}(\mathbb{R}^n)
        \end{equation*}        
        if and only if the two instances are separated by the Tripartite WL-test, and additionally, the variables are correspondingly indexed. Specifically, $C^{T,u_j} = C^{T,\Bar{u}_j}$ must hold for all $j \in [n]$.
    \end{itemize}
\end{proposition}

For the detailed proof of this proposition, see Appendix~\ref{sec:message_passing_GNN_has_equal_separation_power_proof}.

\subsection{Proof of main theorem}

Now we can prove the main theorem. First, we the following proposition.

\begin{proposition}\label{cor:target_function_passing}
    Let $\mathcal{I},\Bar{\mathcal{I}}$ (encoded by $G,\Bar{G}\in\mathcal{G}^{m,n}_{\mathrm{QCQP}}$) be two QCQP instances. If the tripartite WL-test fails to separate the two instances, then the following holds:
    \begin{itemize}
        \item[(i)] If one is feasible, the other is also feasible, i.e., \(\Phi_{\mathrm{feas}}(G) = \Phi_{\mathrm{feas}}(\Bar{G})\).
        \item[(ii)] Assume both instances are feasible. If one is unbounded, the other is also unbounded.
        \item[(iii)] Assume both instances are bounded. Then they have equal optimal values, i.e., \(\Phi_{\mathrm{obj}}(G) = \Phi_{\mathrm{obj}}(\Bar{G})\).
        \item[(iv)] Assume both instances are bounded and that the variables and constraints are indexed such that \(C^{T,u_j} = C^{T,\Bar{u}_j}\). Then they have the same optimal solution, with the least \(\ell_2\)-norm, i.e., \(\Phi_{\mathrm{sol}}(G) = \Phi_{\mathrm{sol}}(\Bar{G})\).
    \end{itemize}
\end{proposition}

\begin{proof}
    \textbf{Passing feasibility}. Assume that \(\mathcal{I}\) is feasible, and let \(x\) be a feasible solution. By Lemma~\ref{lemma:solution_passing}, we obtain another solution \(\Bar{x}\) for instance \(\Bar{\mathcal{I}}\), which implies the feasibility of \(\Bar{\mathcal{I}}\). By switching the roles of \(\mathcal{I}\) and \(\Bar{\mathcal{I}}\), we prove the reverse claim.

    \textbf{Passing unboundedness}. Assume that \(\mathcal{I}\) is unbounded, i.e., for any \(M > 0\), there exists a solution \(x_M\) such that the objective \(f(x) \leq -M\). For each \(x_M\), we can construct a solution \(\Bar{x}_M\) for \(\Bar{\mathcal{I}}\) such that the objective \(\Bar{f}(\Bar{x}_M) \leq f(x_M) \leq -M\), implying that \(\Bar{\mathcal{I}}\) is also unbounded. Again, by switching the roles of \(\mathcal{I}\) and \(\Bar{\mathcal{I}}\), we prove the reverse claim.

    \textbf{Passing optimal values}. Assume that \(\mathcal{I}\) is feasible and bounded, and let \(x\) be its optimal solution. By Lemma~\ref{lemma:solution_passing}, we construct a solution \(\Bar{x}\) for \(\Bar{\mathcal{I}}\) such that:
    \begin{equation*}
        \Bar{f}(\Bar{x}) \leq f(x) = \Phi_{\mathrm{obj}}(G)    
    \end{equation*}
    implying that \(\Phi_{\mathrm{obj}}(\Bar{G}) \leq \Phi_{\mathrm{obj}}(G)\). Similarly, we can show that \(\Phi_{\mathrm{obj}}(G) \leq \Phi_{\mathrm{obj}}(\Bar{G})\), and thus \(\Phi_{\mathrm{obj}}(\Bar{G}) = \Phi_{\mathrm{obj}}(G)\).

    \textbf{Passing optimal solutions}. To prove the last claim, we need the construction of \(\Bar{x}\) from the detailed proof of Lemma~\ref{lemma:solution_passing} (see Appendix~\ref{sec:core_lemma_proof}). Assume that \(\mathcal{I}\) is feasible and bounded, and let \(x\) be its optimal solution (with the least \(\ell_2\)-norm). By Lemma~\ref{lemma:solution_passing}, we construct \(y\) for \(\Bar{\mathcal{I}}\) and \(z\) for \(\mathcal{I}\) by switching the roles of \(\mathcal{I}\) and \(\Bar{\mathcal{I}}\).

    We have \(f(z) \leq \Bar{f}(y) \leq f(x)\) and \(\|z\| \leq \|y\| \leq \|x\|\), which implies that \(z\) is not worse than the given optimal solution \(x\), and thus \(z = x\). By the construction of the averaged solution (and the assumption \(C^{T,u_j} = C^{T,\Bar{u}_j}\)), we have \(y = z\). Combining the two equalities, we conclude that \(x = y\).

    Let \(\Bar{x}\) be the optimal solution of \(\Bar{G}\), and we have \(\|\Bar{x}\| \leq \|y\| = \|x\|\). By switching the roles of \(\mathcal{I}\) and \(\Bar{\mathcal{I}}\), we obtain \(\|x\| \leq \|\Bar{x}\|\), and thus \(\|\Bar{x}\| = \|x\|\). Similarly, we have \(\Bar{f}(\Bar{x}) \leq \Bar{f}(y) \leq f(x)\), and by switching the roles, \(\Bar{f}(\Bar{x}) = f(x)\).

    Since \(\|y\| = \|x\| = \|\Bar{x}\|\) and \(\Bar{f}(y) = f(x) = f(\Bar{x})\), by uniqueness, we conclude that \(y = \Bar{x}\), proving the fourth claim.
\end{proof}

The next step is to extend this separation power to approximation power, which leads to our main theorem. We utilize the generalized Weierstrass-Stone theorem (Theorem 22 and Lemma 36 of~\cite{azizian2020expressive}) and Lusin's theorem.

By applying the generalized Weierstrass-Stone theorem, we establish the following proposition, which demonstrates the approximation power on equivariant functions with compact support.

\begin{proposition}[Uniform Approximation on Continuous Equivariant Functions with Compact Support]\label{prop:uniform_approximation_for_continuous_equivariant_compact_support}
    Let \(\Phi_c: \mathcal{G}^{m,n}_{\mathrm{c}} \to \mathbb{R}^s\) be a general continuous target function defined on a compact subset \(\mathcal{G}_c \subseteq \mathcal{G}^{m,n}_{\mathrm{QCQP}}\), such that:
    \begin{itemize}
        \item If \(s = 1\), the output remains unchanged if the input graph is re-indexed.
        \item If \(s = n\), the output re-indexes accordingly if the input graph is re-indexed.
    \end{itemize}
    
    If the following holds:
    \begin{equation}
        \left( F(G) = F(\Bar{G}), \forall F \in \mathcal{F}_{\mathrm{QCQP}}^{m,n}(\mathbb{R}^s) \Rightarrow \Phi(G) = \Phi(\Bar{G}) \right), \forall G, \Bar{G} \in \mathcal{G}^{m,n}_{\mathrm{c}}
    \end{equation}
    i.e., the family \(\mathcal{F}_{\mathrm{QCQP}}^{m,n}(\mathbb{R}^s)\) separates the target function \(\Phi\), then for any \(\delta > 0\), there exists a function \(F_\delta \in \mathcal{F}_{\mathrm{QCQP}}^{m,n}(\mathbb{R}^s)\) such that:
    \begin{equation}
        \|F_{\delta}(\mathcal{G}) - \Phi(G)\| < \delta 
    \end{equation}
\end{proposition}

For the detailed proof, see Appendix~\ref{sec:uniform_approximation_for_continuous_equivariant_compact_support_proof}. 

However, the requirement for the target function to apply the proposition is too strong. In fact, all target functions defined in~\ref{def:QCQP_properties} are non-continuous and not defined on a compact subset, although equivariance naturally holds. Therefore, we seek a continuous approximation with compact support that can be uniformly approximated. By applying Lusin's theorem, we construct the following continuous approximation:

\begin{proposition}[Continuous Approximation with Compact Support]\label{prop:continuous_compact_support_approximation}
    Let \(\Phi: \mathcal{G}^{m,n}_{\mathrm{QCQP}} \to \mathbb{R}^s\) be a general target function that is measurable under the probability measure \(\mathbb{P}\). For any \(\varepsilon > 0\), there exists a compact subset \(\mathcal{G}^{m,n}_{\mathrm{c}} \subseteq \mathcal{G}^{m,n}_{\mathrm{QCQP}}\), such that \(\mathbb{P}\{G \in \mathcal{G}^{m,n}_{\mathrm{c}}\} > 1 - \varepsilon\), and \(\Phi|_{\mathcal{G}^{m,n}_{\mathrm{c}}}\) is continuous.
\end{proposition}

By combining all the lemmas and propositions, we can now prove the main theorem.

\begin{proof}[Proof of Theorem~\ref{thm:approximation}]
    Let \(\Phi\) be any target function defined in Definition~\ref{def:QCQP_properties}.
    
    By Proposition~\ref{prop:continuous_compact_support_approximation}, \(\Phi\) is continuous on a compact subset \(\mathcal{G}^{m,n}_{\mathrm{c}} \subseteq \mathcal{G}^{m,n}_{\mathrm{QCQP}}\), with \(\mathbb{P}(G \in \mathcal{G}^{m,n}_{\mathrm{c}}) \geq 1 - \frac{\varepsilon}{|\Sigma|}\).

    We construct \(\mathcal{G}^{m,n}_{\mathrm{c,eq}} = \cap_{(\sigma,\tau)\in\Sigma} (\sigma,\tau)(\mathcal{G}^{m,n}_{\mathrm{c}})\). This subset is continuous with compact support, ensuring that \(\Phi|_{\mathcal{G}^{m,n}_{\mathrm{c,eq}}}\) remains an equivariant function, with the following measure control:
    
    \begin{equation}
        \mathbb{P}(G \in \mathcal{G}^{m,n}_{\mathrm{c,eq}}) > 1 - \varepsilon
    \end{equation}
    
    Since by Proposition~\ref{cor:target_function_passing} and the fact that the Tripartite WL-test has equal separation power as the tripartite MP-GNNs, the target functions are equivariant and separated by \(\mathcal{F}_{\mathrm{QCQP}}^{m,n}(\mathbb{R}^s)\). Thus, we may apply Proposition~\ref{prop:uniform_approximation_for_continuous_equivariant_compact_support} and obtain \(F \in \mathcal{F}_{\mathrm{QCQP}}^{m,n}(\mathbb{R}^s)\) such that:
    \[
        \|F(G) - \Phi(G)\| < \delta, \forall G \in \mathcal{G}^{m,n}_{\mathrm{c,eq}}
    \]
    This implies that \(\mathbb{P}\{\|F(G) - \Phi(G)\| < \delta\} > 1 - \varepsilon\).
\end{proof}

\section{Proof of propositions in Section ~\ref{sec:detailed_proof}}

This section provides complete proofs of several propositions in Section~\ref{sec:detailed_proof} that were not immediately proven.

\subsection{Equivariance}\label{sec:equivariance}

We begin by describing equivariance, a key tool used to capture the fact that the indexing of variables and constraints is irrelevant:

\begin{definition}
    Given a function $f:X \to Y$, where $X$ and $Y$ are subsets of Euclidean spaces, and a group $\Sigma$ that acts continuously on $X$ and $Y$, the function $f$ is called equivariant (with respect to the group $\Sigma$) if the following holds:
    \begin{equation*}
        \sigma \circ f(x) = f \circ \sigma(x), \quad \forall x \in X, \sigma \in \Sigma
    \end{equation*}
\end{definition}

Since the indexing of variables and constraints does not affect the problem, we take $\Sigma = S_n \times S_m$, which represents all possible re-indexings of variables and constraints. When applied to both the input and output spaces, we re-index the variables, constraints, and possible solutions (in cases where the output is a solution $x \in \mathbb{R}^n$). Specifically, we have:

\begin{equation*}
    \begin{aligned}
        \tilde{q}_{\pi(j),\pi(k)} & = q_{j,k} \\
        \tilde{p}_{\pi(j)} & = p_j \\
        \tilde{q}^{\tau(i)}_{\pi(j),\pi(k)} & = q^i_{j,k} \\
        \tilde{p}^{\tau(i)}_{\pi(j)} & = p^i_j \\
        \tilde{b}_{\tau(i)} & = b_i \\
        \tilde{x}^{\mathrm{L}}_{\pi(j)} & = x^\mathrm{L}_j \\
        \tilde{x}^{\mathrm{U}}_{\pi(j)} & = x^\mathrm{U}_j \\
    \end{aligned}
\end{equation*}
where the tilde symbols $\tilde{Q}, \tilde{p}, \tilde{b}, \tilde{x}^{\mathrm{L}}, \tilde{x}^{\mathrm{U}}$ denote the re-indexed vectors and matrices.

For $\mathbb{R}^s = \mathbb{R}$, the action on the output space is the identity map: $(\pi, \tau)(\cdot) = \mathrm{id}$. For $\mathbb{R}^s = \mathbb{R}^n$, we correspondingly re-index the output, i.e., $(\pi, \tau)(y)_{\pi(j)} = y_j$.

We can also apply the permutations to:
\begin{itemize}
    \item A point in $\mathbb{R}^n$ (such as a solution), by $(\pi, \tau)(x)_{\pi(j)} = x_j$.
    \item A subset of $\mathbb{R}^n$, by applying the permutation to each element in the subset, or to its indicator function by permuting the underlying set.
\end{itemize}

Equivariance allows us to show that the indices do not matter, while the inputs (in the form of coefficient tuples) necessarily carry these indices.

\textbf{Remark}: Given the group $\Sigma$ and its action on both the input and output, all message-passing layers are automatically equivariant. Thus, requiring $F \in \mathcal{F}_{\mathrm{QCQP}}^{m,n}(\mathbb{R}^s)$ to be equivariant is equivalent to requiring the readout layer $R$ to be equivariant. This is why the readout function must take specific forms in the two cases. While the defined forms do not cover all possible equivariant readout functions, they are general enough to capture the separation power. 

\subsection{Proof of Lemma~\ref{lemma:solution_passing}}\label{sec:core_lemma_proof}

For simplicity of proof, we extend the definitions of $\Phi_{\mathrm{obj}}$ and $\Phi_{\mathrm{sol}}$ to the entire space $\mathcal{G}^{m,n}_{\mathrm{QCQP}}$ by assigning a default value of $0$ (or $\boldsymbol{0}$, depending on the output dimension $s$) when the target function is not defined at a graph $G$. This occurs when the corresponding instance is either infeasible or unbounded, and the optimal value or optimal solution does not exist. By doing so, all target functions are defined on the same space $\mathcal{G}^{m,n}_{\mathrm{QCQP}}$. Moreover, since we approximate feasibility and boundedness, we can distinguish whether the output is the default value or genuinely happens to be $0$ (or $\boldsymbol{0}$).

Let $\mathcal{I}$ and $\Bar{\mathcal{I}}$ be two instances (with Tripartite graph representations $G$ and $\Bar{G} \subseteq \mathcal{G}^{m,n}_{\mathrm{QCQP}}$) that are not separated by the Tripartite WL-test. Without loss of generality, we assume that the variables and constraints are correspondingly indexed, i.e., $C^{T,u_j} = C^{T,\Bar{u}_j}$ and $C^{T,c_i} = C^{T,\Bar{c}_i}$ hold for all $i,j$.

We first introduce the following notations. Let $I$ be any color, and we collect all nodes of a graph $G$ with color $I$, denoting this collection as $G(I)$. Throughout this paper, we use $J$ for the colors of variable nodes, $K$ for quadratic nodes, and $I$ for constraint nodes.

We now present the following lemma:

\begin{lemma}\label{lem:stable_color_sum}
    Given the graph $G$, let the Tripartite WL-test stabilize after $T \geq 0$ iterations. The sum of weights from a certain node of one color to all nodes of another color depends only on the color of the given node. Specifically, the sum (taking $J$ for variable nodes and $K$ for quadratic nodes as an example) is:
    \begin{equation*}
        S(J,K;G) := \sum_{C^{T,v}=K} w_{u,v}
    \end{equation*}
    and is well-defined with $u \in G(J)$ arbitrarily chosen.

    Similarly, for any color of constraints $I$, color of variables $J$, and color of quadratic terms $K$, the following sums are well-defined:
    \begin{equation*}
        \begin{split}
            S(J,I;G) & := \sum_{C^{T,c}=I} w_{u,c}, \quad C^{T,u}=J \\
            S(I,K;G) & := \sum_{C^{T,v}=K} w_{c,v}, \quad C^{T,c}=I \\
            S(K,I;G) & := \sum_{C^{T,c}=I} w_{v,c}, \quad C^{T,v}=K \\
            S(J,K;G) & := \sum_{C^{T,v}=K} w_{u,v}, \quad C^{T,u}=J \\
            S(K,J;G) & := \sum_{C^{T,u}=J} w_{v,u}, \quad C^{T,v}=K \\
        \end{split}
    \end{equation*}
\end{lemma}

\begin{proof}
    Let $v, v'$ be two nodes with color $K = C^{T,v} = C^{T,v'}$. Since the Tripartite WL-test has stabilized, further iterations do not separate additional node pairs, i.e.,
    \begin{equation*}
        \sum_u w_{u,v}\mathrm{HASH}(C^{T,u}) = \sum_u w_{u,v'}\mathrm{HASH}(C^{T,u}).
    \end{equation*}
    Rearranging according to $J = C^{T,u}$, we get:
    \begin{equation*}
        \sum_J \sum_{C^{T,u}=J} w_{u,v} \cdot \mathrm{HASH}(J) = \sum_J \sum_{C^{T,u}=J} w_{u,v'} \cdot \mathrm{HASH}(J).
    \end{equation*}
    Assuming that the hash function is collision-free, we conclude that:
    \begin{equation*}
        \sum_{C^{T,u}=J} w_{u,v} = \sum_{C^{T,u}=J} w_{u,v'},
    \end{equation*}
    i.e., $S(K,J;G) := \sum_{C^{T,u}=J} w_{v,u}, \quad C^{T,v}=K$ is well-defined.

    The other claims follow similarly.
\end{proof}

By summing all weights between two colors $I$ and $J$, we derive the following lemma:

\begin{lemma}\label{lem:stable_color_equation}
    Let $J$ and $K$ be arbitrary node colors. Then, the following holds:
    \begin{equation*}
        |G(J)|S(J,K;G) = |G(K)|S(K,J;G),
    \end{equation*}
    and similar equalities hold between $I$ and $J$, and between $I$ and $K$.
\end{lemma}

\begin{proof}
    Summing all edges between all nodes with $C^{T,u}=J$ and $C^{T,v}=K$, and re-arranging the sum according to $u$ and $v$, by Lemma~\ref{lem:stable_color_sum}, we have:
    \begin{equation*}
        |G(J)|S(J,K;G) = |G(K)|S(K,J;G).
    \end{equation*}

    The other two claims are similar.
\end{proof}

We are now ready to proceed. We construct $\Bar{x}_j = \frac{1}{|G(J)|} \sum_{j':C^{T,u_{j'}}=C^{T,\Bar{u}_{j'}}=J} x_{j'}$, where $J=C^{T,x_j}$. We claim that $\Bar{x}$ satisfies all the required conditions.

First, we analyze the \textbf{linear part} of the constraints and the objective. Let $f^i_{\mathrm{lin}}(x):= p^i \cdot x$ represent the linear part of the $i$-th constraint. For a certain color $I$ of constraint nodes, we have:
\begin{equation}
    \begin{split}
        \Bar{f}^i_{\mathrm{lin}}(\Bar{x}) & = \sum_j \Bar{p}^i_j \Bar{x}_j \\
        & = \sum_J \sum_{\bar{v}_j \in G(J)} \Bar{p}^i_j \Bar{x}_j \\
        & = \sum_J S(I,J) \Bar{x}_J \\
        & = \frac{1}{|G(I)|} \sum_J S(J,I) |G(J)| \Bar{x}_J \\
        & = \frac{1}{|G(I)|} \sum_J S(J,I) \sum_{u_j \in G(J)} x_j \\
        & = \frac{1}{|G(I)|} \sum_{c_i \in G(I)} \sum_J \sum_{j \in G(J)} p^i_j x_j \\
        & = \frac{1}{|G(I)|} \sum_{c_i \in G(I)} f^i_{\mathrm{lin}}(x).
    \end{split}
\end{equation}

Here, $\Bar{x}_j$ is the average over the nodes with color $J$, so it is determined by $J$, and we denote its value as $\Bar{x}_J$.

We define $f_{\mathrm{lin}}(x) = p \cdot x$. For the objective part, we have:
\begin{equation}
    \begin{split}
        \sum_j \Bar{p}_j \Bar{x}_j & = \sum_J p_J |\Bar{G}(J)| \Bar{x}_J \\
        & = \sum_J p_J \sum_{u_j \in G(J)} x_j \\
        & = \sum_j p_j x_j,
    \end{split}
\end{equation}
where $p_j, \Bar{p}_j$ are the features of the variables, which are determined by the color $J = C^{T,u_j} = C^{T,\Bar{u}_j}$. We denote this value by $p_j = \Bar{p}_j = p_J$.

\textbf{Quadratic part}. We define $f^i_{\mathrm{quad}}(x) = \frac{1}{2} x^\top Q^i x$ as the \textbf{quadratic} part of the $i$-th constraint.

For a certain color $I$ of constraint nodes, we have the following:
\begin{equation*}
    \begin{split}
        \Bar{f}^i_{\mathrm{quad}}(\Bar{x}) & = \frac{1}{2} \sum_{v_{j,k} \in V_2(\Bar{G})} \Bar{q}^i_{j,k} \Bar{x}_j \Bar{x}_k \\
        & = \frac{1}{2} \sum_K \sum_{\Bar{v}_{j,k} \in \Bar{G}(K)} \Bar{q}^i_{j,k} \Bar{x}_j \Bar{x}_k \\
        & = \frac{1}{2} \sum_K S(I,K) |\Bar{G}(K)| \Bar{x}_K.
    \end{split}
\end{equation*}
Since all $\Bar{v}_{j,k} \in V_2(\Bar{G})$ have $\Bar{u}_j, \Bar{u}_k$ as neighbors in $V_1(\Bar{G})$, $\Bar{x}_K := \Bar{x}_j \Bar{x}_k$ is well-defined. This equation shows that the value $\Bar{f}^i_{\mathrm{quad}}(\Bar{x})$ depends only on the color $I = C^{T,\Bar{c}_i}$, and not on the specific selection of $\Bar{c}_i \in \Bar{G}(I)$. Therefore, $f^i_{\mathrm{quad}}(\Bar{x})$ reduces to the sum, and we claim that $f^i_{\mathrm{quad}}(\Bar{x}) = \Bar{f}^i_{\mathrm{quad}}(\Bar{x})$ holds.

Next, we consider the partial derivative. Let $J := C^{T,u_j}$, and we have:
\begin{equation}
    \begin{split}
        \partial_j \sum_{c_i \in G(I)} f^i_{\mathrm{quad}}(\Bar{x}) & = \sum_{c_i \in G(I)} \sum_k w(u_j, v_{j,k}) w(v_{j,k}, c_i) \Bar{x}_k \\
        & = \sum_{c_i \in G(I)} \sum_K \sum_{k:v_{j,k} \in G(K)} w(u_j, v_{j,k}) w(v_{j,k}, c_i) \Bar{x}_k \\
        & = \sum_K S(K, I) \sum_{k:v_{j,k} \in G(K)} w(u_j, v_{j,k}) \Bar{x}_k \\
        & = \sum_K S(K, I) S(J, K) x_{K;J}.
    \end{split}
\end{equation}
Since $u_j$ is one of the neighbors in $v_{j,k}$, and $v_{j,k} \in G(K)$ has exactly two neighbors in $V_1(G)$, we know that the color of $u_k$ depends only on the colors $K = C^{T,v_{j,k}}$ and $J = C^{T,v_j}$. This makes $\Bar{x}_{K;J} := \Bar{x}_k$ well-defined, with $u_j \in G(J)$ and $v_{j,k} \in G(K)$.

Thus, the derivative $\partial_j \sum_{c_i \in G(I)}$ depends only on $J = C^{T,u_j}$, i.e.,
\begin{equation}\label{eq:equal_derivative}
    C^{T,v_{j_1}} = C^{T,v_{j_2}} \Rightarrow \partial_{j_1} \sum_{c_i \in G(I)} f^i_{\mathrm{quad}}(\Bar{x}) = \partial_{j_2} \sum_{c_i \in G(I)} f^i_{\mathrm{quad}}(\Bar{x}).
\end{equation}

By \eqref{eq:equal_derivative}, we know that $\Bar{x}$ is a local optimal point within the linear space:
\begin{equation*}
    \{y \in \mathbb{R}^n : \sum_{u_j \in G, C^{T,u_j} = J} y_j = \sum_{u_j \in G, C^{T,u_j} = J} x_j\}.
\end{equation*}

With the convexity assumption, the local optimal point is a global minimum. Since $x$ is in this linear space, we claim that:
\begin{equation}\label{eq:average_sol_leq}
    \sum_{c_i \in G(I)} f^i_{\mathrm{quad}}(\Bar{x}) \leq \sum_{c_i \in G(I)} f^i_{\mathrm{quad}}(x).
\end{equation}

Combining \eqref{eq:average_sol_leq} with the fact that $f^i_{\mathrm{quad}}(\Bar{x})$ and $\Bar{f}^i_{\mathrm{quad}}(\Bar{x})$ are equal for all $c_i \in G(I)$, we can control the quadratic parts:
\begin{equation}
    \begin{split}
        \Bar{f}^i_{\mathrm{quad}}(\Bar{x}) & = f^i_{\mathrm{quad}}(\Bar{x}) \\
        & = \frac{1}{|G(I)|} \sum_{c_i \in G(I)} f^i_{\mathrm{quad}}(\Bar{x}) \\
        & \leq \frac{1}{|G(I)|} \sum_{c_i \in G(I)} f^i_{\mathrm{quad}}(x).
    \end{split}
\end{equation}

For the objective part, we define $f_{\mathrm{quad}}(x) = \frac{1}{2} x^\top Q x$. Similarly, we have:
\begin{equation*}
    \begin{split}
        \Bar{f}_{\mathrm{quad}}(\Bar{x}) & = \frac{1}{2} \sum_{\Bar{v}_{j,k}} f^0(\Bar{v}_{j,k}) \Bar{x}_j \Bar{x}_k \\
        & = \frac{1}{2} \sum_K \sum_{\Bar{v}_{j,k} \in \Bar{G}(K)} f^0(\Bar{v}_{j,k}) \Bar{x}_j \Bar{x}_k \\
        & = \frac{1}{2} \sum_K \sum_{v_{j,k} \in G(K)} f^0(v_{j,k}) \Bar{x}_j \Bar{x}_k \\
        & = f_{\mathrm{quad}}(\Bar{x}).
    \end{split}
\end{equation*}

We also have:
\begin{equation*}
    \begin{split}
        \partial_j f_{\mathrm{quad}}(\Bar{x}) & = \sum_k f^0(v_{j,k}) w(u_j, v_{j,k}) \Bar{x}_k \\
        & = \sum_K \sum_{k:v_{j,k} \in G(K)} f^0(v_{j,k}) w(u_j, v_{j,k}) \Bar{x}_k \\
        & = \sum_K f^0(K) S(J, K) \Bar{x}_{K;J},
    \end{split}
\end{equation*}
which depends only on $J = C^{T,u_j}$. Here, $f^0(K) = f^0(v_{j,k})$, and $v_{j,k} \in G(K)$ is well-defined by the stable color assumption.

\textbf{Combination of the two parts.}

The color $C^{T,c_i} = C^{T,\Bar{c}_i} = I$ determines the RHS $b_I := b_i$. Defining $f^i_{\mathrm{cons}}(x) = f^i_{\mathrm{quad}}(x) + f^i_{\mathrm{lin}}(x)$, and similarly for $\Bar{\mathcal{I}}$, we have:
\begin{equation*}
    \begin{split}
        \bar{f}^i_{\mathrm{cons}}(\Bar{x}) & = \Bar{f}^i_{\mathrm{quad}}(\Bar{x}) + \Bar{f}^i_{\mathrm{lin}}(\Bar{x}) \\
        & = \frac{1}{|G(I)|} \sum_{c_i \in G(I)} \left( f^i_{\mathrm{quad}}(x) + f^i_{\mathrm{lin}}(x) \right) \\
        & = \frac{1}{|G(I)|} \sum_{c_i \in G(I)} f^i_{\mathrm{cons}}(x) \\
        & \leq b_I.
    \end{split}
\end{equation*}

For the objective, we similarly have:
\begin{equation*}
    \Bar{f}_{\mathrm{quad}}(\Bar{x}) + \Bar{f}_{\mathrm{lin}}(\Bar{x}) \leq f_{\mathrm{quad}}(x) + f_{\mathrm{lin}}(x).
\end{equation*}

This completes the proof that $\Bar{x}$ is the solution for $\Bar{\mathcal{I}}$, satisfying the condition given in Proposition~\ref{lemma:solution_passing}.

\subsection{Proof of Proposition~\ref{prop:message_passing_GNN_has_equal_separation_power}}\label{sec:message_passing_GNN_has_equal_separation_power_proof}

We prove the separation power by simulating the tripartite WL-test using tripartite MP-GNNs. We define the hidden representation $h^{t,\cdot}$, produced by some network, as a \textbf{one-hot representation} of the colors $C^{t,\cdot}$ if all $h^{t,\cdot}$ are one-hot vectors, and they take the same value if and only if they have the same color $C^{t,\cdot}$.

First, we consider the color initialization. We collect all the features paired with the node types (i.e., variable nodes, quadratic nodes, and constraint nodes). Then we select $g^0_{1,2,3}$ to map the features to one-hot vectors, where the enumeration serves as the only index with the value $1.0$. For example, if the feature $h^{u_j}$ of a variable node is enumerated by $r$, then $g^0_1$ maps $h^{u_j}$ to $h^{0,u_j}=e_r$.

It's easy to see that the embedded hidden feature $h^{0,\cdot}$ is a one-hot representation of the initial color $C^{0,\cdot}$.

Next, we consider the first refinement. Assuming that $h^{t,\cdot}$ is a one-hot representation of $C^{t,\cdot}$ and $g^t_1 = \mathrm{id}$ is a simple and proper hash function, the concatenated vector
\[
\left[h^{t,v}, \sum_{u \in V_1} w_{u,v} f_1^t(h^{t,u})\right]
\]
is a representation of the colors $\Bar{C}^{t,\cdot}$, which is generally not one-hot. The same holds for the other three concatenated vectors from the remaining three sub-layers. By Theorem 3.2 of \citealp{yun2019small}, a network with four fully connected layers and ReLU activation maps these values back to one-hot. Therefore, we select $f^t_1$ to concatenate the inputs and then pass them through a 4-layered MLP with ReLU activation, so that the aggregated hidden representation $\Bar{h}^{t,\cdot}$ is once again one-hot.

Similarly, we get $h^t_{2,3,4}$ and $g^t_{2,3,4,5,6}$ and simulate an iteration of the Tripartite WL-test with a round of four message-passing sub-layers.

In the case of graph-level output, the readout function takes the following form:
\[
R(\cdot) = f_{\mathrm{out}}\left( \sum_j h^{T,u_j}, \sum_{j,k} h^{T,v_{j,k}}, \sum_i h^{T,c_i} \right).
\]
Since the hidden representation is a one-hot representation of $C^{T,\cdot}$, if two instances are not separated by the tripartite message-passing GNN, they are not separated by this subset of GNNs (given a fixed initialization and a free readout function). Consequently, all entries must be equal, and the two instances are not separated by the Tripartite WL-test.

Similarly, in the case of node-level output, all equivariant readout functions take the form:
\[
R(\cdot)_j = f_{\mathrm{out}}\left(h^{T,u_j}, \sum_j h^{T,u_j}, \sum_{j,k} h^{T,v_{j,k}}, \sum_i h^{T,c_i}\right).
\]
Thus, all entries must be equal, and the two instances are not separated. Moreover, the variables are correspondingly indexed.

Conversely, we use induction to prove that for all $t \in \mathbb{N}$, the colors $C^{t,\cdot}$ separate more than the hidden features $h^{t,\cdot}$, i.e.,
\begin{equation}
    C^{t,u} = C^{t,u'} \Rightarrow h^{t,u} = h^{t,u'}, \quad \forall u, u' \in V_1 \cup \Bar{V}_1, F \in \mathcal{F}_{\mathrm{QCQP}}^{m,n}(\mathbb{R}^s),
\end{equation}
and similar claims hold for the other three sub-iterations.

For $t=0$ (i.e., right after embedding), the statement is obviously true. Now, assume that after some sub-iteration (say, before the first sub-iteration of iteration $t \geq 1$, with the other sub-iterations following similarly), the statement holds.

Let $v, v'$ satisfy:
\[
\sum_{u} w_{u,v} \mathrm{HASH}(C^{t,u}) = \sum_{u} w_{u,v'} \mathrm{HASH}(C^{t,u}).
\]
Organizing the sum by $C^{t,u} = J$, and assuming the hash function is collision-free, we have:
\begin{equation}\label{eq:proof_equal_separation_power_1}
    \sum_{u: C^{t,u} = J} w_{u,v} = \sum_{u: C^{t,u} = J} w_{u,v'}, \quad \forall J.
\end{equation}

Next, we organize the sum $\sum_{u} w_{u,v} f_1^t(h^{t,u})$ by the value of $h^{t,u}$. By the induction assumption, the set $\{u: h^{t,u} = h\}$ is the union of $\{u: C^{t,u} = J_l\}$ for some colors $J_l$. Summing the equality in~\eqref{eq:proof_equal_separation_power_1} over the colors, we have:
\[
\sum_{h^{t,u} = h} w_{u,v} = \sum_{h^{t,u} = h} w_{u,v'}, \quad \forall h.
\]
Thus, we conclude:
\[
\sum_{u} w_{u,v} f_1^t(h^{t,u}) = \sum_h \sum_{u: h^{t,u} = h} w_{u,v} f_1^t(h) = \sum_h \sum_{u: h^{t,u} = h} w_{u,v'} f_1^t(h) = \sum_{u} w_{u,v'} f_1^t(h^{t,u}),
\]
which completes the induction.

For the case of graph-level output, this means that all entries of the input to the readout function are equal for the two graphs, i.e.,
\[
\left(\sum_j h^{T,u_j}, \sum_{j,k} h^{T,v_{j,k}}, \sum_i h^{T,c_i}\right) = \left(\sum_j h^{T,\Bar{u}_j}, \sum_{j,k} h^{T,\Bar{v}_{j,k}}, \sum_i h^{T,\Bar{c}_i}\right),
\]
and the GNNs give the same output for all possible readout functions.

For the case of node-level output, we again have:
\[
\left(h^{T,u_j}, \sum_j h^{T,u_j}, \sum_{j,k} h^{T,v_{j,k}}, \sum_i h^{T,c_i}\right) = \left(h^{T,\Bar{u}_j}, \sum_j h^{T,\Bar{u}_j}, \sum_{j,k} h^{T,\Bar{v}_{j,k}}, \sum_i h^{T,\Bar{c}_i}\right).
\]
Here, we use the assumption that the variables are correspondingly indexed to guarantee $h^{T,u_j} = h^{T,\Bar{v}_j}$.

\subsection{Proof of Proposition~\ref{prop:uniform_approximation_for_continuous_equivariant_compact_support}}\label{sec:uniform_approximation_for_continuous_equivariant_compact_support_proof}

The requirement for the general target function $\Phi_c$ is simply equivariance under re-indexing. Thus, we need to verify the conditions required by the generalized Weierstrass theorem (Theorem 22 of~\cite{azizian2020expressive}) to apply.

First, we verify that $\mathcal{F} = \mathcal{F}^{m,n}_{\mathrm{QCQP}}(\mathbb{R}^s)$ is a sub-algebra. By multiplying the readout function by $\lambda$, we construct $\lambda F \in \mathcal{F}_{\mathrm{QCQP}}^{m,n}(\mathbb{R}^s)$. Now, we construct the sum and product of two functions $F_1, F_2 \in \mathcal{F}_{\mathrm{QCQP}}^{m,n}(\mathbb{R}^s)$.

Given $F_1$ and $F_2$, we proceed as follows:
\begin{itemize}
    \item We construct
    \[
        g^0_{1,F}(h^{0,u}) := \left[g^0_{1,F_1}(h^{0,u}), g^0_{1,F_2}(h^{0,u})\right].
    \]
    We give similar constructions for $g^0_{2,F}$ and $g^0_3$.
    \item After initialization, all hidden features take the form $h^{t,u} = [h^{t,u}_{F_1}, h^{t,u}_{F_2}]$ (considering variable nodes as an example, and similarly for quadratic nodes). We construct
    \[
        g^t_{1,F}(h^{t,u}) := \left[g^t_{1,F_1}(h^{t,u}_{F_1}), g^t_{1,F_2}(h^{t,u}_{F_2})\right],
    \]
    and
    \[
        f^t_{1,F}(h^{t,v}, \sum_u w_{uv} h^{t,u}) := \left[f^t_{1,F_1}(h^{t,v}_{F_1}, \sum_u w_{uv} h^{t,u}_{F_1}), f^t_{1,F_2}(h^{t,v}_{F_2}, \sum_u w_{uv} h^{t,u}_{F_2})\right].
    \]
    We give similar constructions for other $g^t_{\cdot,F}, f^t_{\cdot,F}$. Using this construction, we compute both hidden representations in one concatenated network.
    \item Finally, we obtain $F = F_1 + F_2$ by constructing $R(\cdot) = R_1(\cdot_{F_1}) + R_2(\cdot_{F_2})$, and similarly for $F = F_1 \times F_2$.
\end{itemize}

Thus, we conclude that $F_1 + F_2, F_1 \times F_2 \in \mathcal{F}_{\mathrm{QCQP}}^{m,n}(\mathbb{R}^s)$.

Next, we verify the inclusion $\rho(\mathcal{F}_{\mathrm{scal}}) \subseteq \rho(\pi_{\Sigma} \circ \mathcal{F})$:

\textbf{Graph-level output case}. In this case, we have $\mathcal{F}_{\mathrm{scal}} = \mathcal{F}$ and $\pi_{\Sigma} = \mathrm{id}$, so the two sides are exactly the same.

\textbf{Node-level output case}. Given any $R_1$ that maps the final hidden representation to a graph-level output, $R_1 \cdot \mathbf{1}_n = (R_1, R_1, \dots, R_1)$ is a valid equivariant readout function in the node-level case. Thus, given any $F \in \mathcal{F}_{\mathrm{QCQP}}^{m,n}(\mathbb{R})$, we can construct $F' \in \mathcal{F}_{\mathrm{QCQP}}^{m,n}(\mathbb{R}^n)$ using $R_1 \cdot \mathbf{1}_n$, along with all the $f$ and $g$ functions, and conclude that any pair $(G_1, G_2) \in \rho(\mathcal{F}_{\mathrm{scal}})$ is not separated by the Tripartite WL-test.

For any pair of graphs $(G, \Bar{G})$ that is not separated by the Tripartite WL-test, after re-indexing variables and constraints, all $F \in \mathcal{F}$ map them to the same output. This means that, without re-indexing, all $F \in \mathcal{F}$ map the two graphs to outputs that differ at most by a re-indexing. Thus, $(G, \Bar{G})$ is contained in $\rho(\pi_{\Sigma} \circ \mathcal{F})$. This completes our verification.

Applying the Generalized Weierstrass-Stone theorem to the sub-algebra $\mathcal{F} = \mathcal{F}_{\mathrm{QCQP}}^{m,n}(\mathbb{R}^s)$ completes the proof.

\section{Proof of propositions in Section~\ref{sec:non_convex_counter_examples}}\label{sec:GNN_not_enough_separation_power_on_non_convex_QCQPs_proof}

The two instances are {QCQP} instances. Both graphs $G$ and $\Bar{G}$ consist of the following:

\begin{itemize}
    \item 6 variable nodes, i.e., $u_j$ or $\Bar{u}_j$, where $j \in [6]$. All nodes carry the feature $h^{u_j} = (0, -1, 1)$. here we assume that $-1\le x_i\le 1$ by the unit ball constraint. 
    \item 12 effective quadratic nodes. The squared nodes carry $h^{v_{j,j}}=(0)$, while others carry the feature $h^{v_{j,k}} = (1)$.
    \item 1 constraint node $c$ representing the unit ball constraint. The node carries feature $(-1)$ for both graphs.
\end{itemize}

We now verify that the Tripartite WL-test does not separate the two graphs:
\begin{itemize}
    \item After initialization, we have $h^0_1 := h^{0,u} = h^{0,\Bar{u}} = \mathrm{HASH}_1((0, -1, 1))$, $h^0_2 := h^{0,v_{j,j}} = h^{0,\Bar{v}} = \mathrm{HASH}_2((0))$, $h^0_3:=h^{0,v_{j,k}}=\mathrm{HASH}_2((1))$ and $h^0_4 := h^{0,c} = h^{0,\Bar{c}} = \mathrm{HASH}_3((-1))$.
    \item After the first sub-iteration, we have
    \[
        \Bar{h}^0_2 := \Bar{h}^{0,v_{j,k}} = \mathrm{HASH}(h^0_2, 2h^0_1),
    \]
    and 
    \[
        \Bar{h}^0_3 := \Bar{h}^{0,v_{j,k}} = \mathrm{HASH}(h^0_3, 2h^0_1),
    \]
    which remains equal for all $v \in V_2(G)$ and $\Bar{v} \in V_2(\Bar{G})$.
    \item After the second sub-iteration, we have
    \[
        h^1_4 := h^{1,c} = h^{1,\Bar{c}} = \mathrm{HASH}(h^0_4, 0, 1 \cdot \Bar{h}^0_2),
    \]
    which remains equal for both graphs.
    \item After the third sub-iteration, we have
    \[
        h^1_2 := h^{1,v_{j,j}} = \mathrm{HASH}(\Bar{h}^0_2, 1 \cdot h^1_3),
    \]
    and
    \[
        h^1_3 := h^{1,v_{j,k}} = \mathrm{HASH}(\Bar{h}^0_3, 0),
    \]
    which remains equal for both graphs.
    \item After the final sub-iteration, we have
    \[
        h^1_1 := h^{1,u} = h^{1,\Bar{u}} = \mathrm{HASH}(h^0_1, 0, 2\cdot h^1_2+1\cdot h^1_3),
    \]
    which remains equal for both graphs.
    \item The Tripartite WL-test terminates after one iteration since no further node pairs are separated.
\end{itemize}

The Tripartite WL-test returns $C^{0,\cdot}$, which is the same for both instances. Thus, we conclude that the two graphs are not separated, with variables and constraints correspondingly indexed. By Proposition~\ref{prop:message_passing_GNN_has_equal_separation_power}, we conclude that, in both the node-level and graph-level cases, tripartite MP-GNNs cannot separate the two instances.

Therefore, we conclude that tripartite MP-GNNs cannot approximate the optimal solution or optimal value for non-convex QCQP instances (even QP instances). To demonstrate that tripartite MP-GNNs cannot accurately predict feasibility, we slightly modify the two instances:

\begin{proof}[Proof of Proposition \ref{prop:can_not_represent_non_convex}]
    We reconstruct the objective as a constraint. Specifically, consider the following two instances:
    \begin{equation}\label{eq:non_convex_instance_1f}
        \begin{aligned}
            & \underset{x \in \mathbb{R}^6}{\min}  & & 0 \\
            & \text{s.t.} & & x_1 x_2 + x_2 x_3 + x_3 x_1 + x_4 x_5 + x_5 x_6 + x_6 x_4 \leq -\frac{3}{4} \\
            & & & \sum_{i=1}^{6} x_i^2 \le 1 \\
        \end{aligned}
    \end{equation}
    and
    \begin{equation}\label{eq:non_convex_instance_2f}
        \begin{aligned}
            & \underset{x \in \mathbb{R}^6}{\min}  & & 0 \\
            & \text{s.t.} & & x_1 x_2 + x_2 x_3 + x_3 x_4 + x_4 x_5 + x_5 x_6 + x_6 x_1 \leq -\frac{3}{4} \\
            & & & \sum_{i=1}^{6} x_i^2 \le 1 \\
        \end{aligned}
    \end{equation}
    Clearly, instance \ref{eq:non_convex_instance_1f} is not feasible, while instance \ref{eq:non_convex_instance_2f} is feasible.

    In the graph generated by the Tripartite graph representation, we change the objective to another special constraint and add a new dummy objective. Similarly, we see that tripartite MP-GNNs fail to separate $\mathcal{I}$ and $\Bar{\mathcal{I}}$.
\end{proof}

\section{Additional experiments}
\label{appendix:add_res}

{

\paragraph{Evaluation on the QPLIB dataset.}

We incorporated a dataset derived from real-world instances in QPLib. Due to computational constraints, we used all instances in QPLIB that are convex and have no more than 5,000 nonzeros as training sets, and augmented these instances by perturbations. Below, we report the training and validation losses in Table \ref{table:train_loss_qplib} and Table \ref{table:valid_loss_qplib}, respectively.

\begin{table}[h]
\centering
\caption{Training loss vs. numbers of parameters on QPLIB.}

\begin{tabular}{lccccccc}
\toprule
\multirow{2}{*}{Target} & \multicolumn{5}{c}{\# Parameters}          \\ \cmidrule{2-6} 
                             & 15K    & 42K    & 126K    & 450K   & 1.7M       \\
\midrule
objective                  & 4.3306  & 0.9986  & 0.8679   & 0.7999  & 0.6614  \\
solution                   & 19.2789 & 18.9083 & 18.8303  & 18.7472 & 18.6495\\
\bottomrule
\end{tabular}
\label{table:train_loss_qplib}
\end{table}

\begin{table}[h]
\centering
\caption{Validation loss vs. numbers of training samples on QPLIB.}

\begin{tabular}{lccccc}
\toprule
\multirow{2}{*}{Target} &        \multicolumn{5}{c}{\# Samples} \\
\cmidrule{2-6}  
                           & 100     & 300     & 500       & 700     & 1000      \\
                           \midrule
objective                  & 1.5414  & 0.8922  & 0.8564    & 0.6176  & 0.5598    \\
solution                   & 19.9106 & 19.7552 & 19.6981   & 19.6788 & 19.6165  \\
\bottomrule
\end{tabular}
\label{table:valid_loss_qplib}
\end{table}

As shown, both training and validation losses improve consistently with increased model capacity and training samples. Such a trend supports the method’s effectiveness on real-world QCQP instances.

}

{

\paragraph{Evaluation on boundedness.}

We further evaluated performance on the task of predicting the boundedness of convex QCQPs. Specifically, unbounded cases arise when there exists a direction $d$ such that $Qd=0$ and $p^\top d<0$ for both the objective and constraints. To examine performance under these challenging edge cases, we generated a dedicated dataset by randomly enforcing such conditions. Using the same training setup as in Section~\ref{sec:main-results}, the corresponding training and validation losses are reported in Tables~\ref{table:train-boundedness} and~\ref{table:valid-boundedness}.

\begin{table}[h]
\centering

\caption{Training loss vs. number of parameters on predicting boundedness.}
\begin{tabular}{lccccc}
\toprule
{\#Params} & 21k    & 42k    & 126k   & 1.7M   & 6.7M   \\
\midrule
Train loss                  & 0.6933 & 0.2504 & 0.1241 & 0.0470 & 0.0427 \\
\bottomrule
\end{tabular}
\label{table:train-boundedness}
\end{table}

\begin{table}[h]
\centering

\caption{Validation loss vs. number of training samples on predicting boundedness.}
\begin{tabular}{lccccc}
\toprule
\multicolumn{1}{c}{\#Samples} & 100    & 300    & 500    & 700    & 1000   \\
\midrule
Valid loss                   & 0.6876 & 0.3495 & 0.2786 & 0.2360 & 0.1285\\
\bottomrule
\end{tabular}
\label{table:valid-boundedness}
\end{table}

The results show a clear trend: training loss decreases with larger model sizes, and validation loss improves with more training data. These results suggest that our approach is capable of learning effectively on unbounded problem instances.

}

\end{document}